\documentclass[letterpaper]{article}
\usepackage[english]{babel}
\usepackage{amsthm}
\usepackage{amsmath, color}
\usepackage{amsfonts}
\usepackage{amssymb}
\usepackage{enumerate}
\usepackage{float}
\usepackage{wrapfig}
\usepackage{graphicx}
\usepackage[thinc]{esdiff}
\usepackage[affil-it]{authblk}
\usepackage{fancyhdr}
\usepackage[thinc]{esdiff}
\usepackage{authblk}

\usepackage{colortbl}

\usepackage[hyperindex,breaklinks,linktocpage=true,plainpages=false,pdfpagelabels=false]{hyperref}
\usepackage{cite}

\definecolor{linkcolor}{rgb}{0, 0, 1}
\definecolor{citecolor}{rgb}{0, 0, 1}
\definecolor{urlcolor}{rgb}{0, 0, 1}
\hypersetup{
	linktoc = all, colorlinks, linkcolor={linkcolor},
	citecolor={citecolor}, urlcolor={urlcolor}
}

   \newcommand{\f}{\varphi}
\newcommand{\ms}{\medskip}
\newcommand{\ssk}{\smallskip}
\newcommand{\q}{\quad}  \newcommand{\qq}{\qquad}
    \newcommand{\const}{\mbox{const}\,}
\newcommand{\R}{\mathbb{R}}   \newcommand{\wt}{\widetilde}

\newcommand{\p}{\psi}   \newcommand{\e}{\varepsilon}

\newcommand{\beq}{\begin{equation}\label}  \newcommand{\eeq}{\end{equation}}
\begin{document}

\title{On the Relation Between Two Approaches \\ to Necessary Optimality
	Conditions in Problems \\ with State Constraints}

\author[1,2]{Andrei Dmitruk\thanks{\href{mailto:dmitruk@member.ams.org}{dmitruk@member.ams.org}
	}}
\author[2]{Ivan Samylovskiy\thanks{\href{mailto:ivan.samylovskiy@cs.msu.ru}{ivan.samylovskiy@cs.msu.ru}
		}}

\affil[1]{Central Economics and Mathematics Institute of the 					  Russian Academy of Sciences}
\affil[2]{Lomonosov Moscow State University, Faculty of 						  Computational Mathematics and Cybernetics}

\renewcommand\Authands{ and }

\newtheorem{theorem}{Theorem}[section]
\newtheorem{prop}{Proposition}[section]
\newtheorem{rem}{Remark}[section]
\newtheorem{definition}{Definition}[section]
\newtheorem{Note}{Note}[section]
\newtheorem{lemma}{Lemma}[section]
\newtheorem{remark}{Remark}[section]
\newtheorem{acknowledgements}{Acknowledgements}[section]
\maketitle

\begin{abstract}
We consider a class of optimal control problems with a state constraint
and investigate a trajectory with a single boundary interval (subarc).
Following R.V.~Gamkrelidze, we differentiate the state constraint
along the boundary subarc, thus reducing the original problem to a
problem with mixed control-state constraints, and show that this way allows
one to obtain the full system of stationarity conditions in the form of
A.Ya. Dubovitskii and A.A. Milyutin, including the sign definiteness of
the measure (state constraint multiplier), i.e. the non-negativity of its
density and atoms at junction points.
The stationarity conditions are obtained by a two-stage variation approach,
proposed in this paper. At the first stage, we consider only those variations,
which do not affect the boundary interval, and obtain optimality conditions
in the form of Gamkrelidze. At the second stage, the variations are concentrated
on the boundary interval, thus making possible to specify the stationarity
conditions and obtain the sign of density and atoms of the measure.

\end{abstract}

\section{Introduction}
It is a well-known fact that optimality conditions in problems with state
constraints are difficult for application in view of a nonstandard character
of the state constraint multiplier. In their seminal work \cite{DM65},
A.Ya. Dubovitskii and A.A. Milyutin  suggested to take this multiplier
in the form of non-negative measure concentrated on the boundary set of
the optimal trajectory (see also later works \cite{Gir,IT}).
This corresponds to the functional meaning of the state constraint, but
then the adjoint equation contains a measure (more precisely, its
generalized derivative
\footnote{
	For a function $\mu(t)$ of bounded variation, its generalized derivative
	$\dot\mu(t)= d\mu(t)/dt$ is a generalized function in the sense that
	$\dot\mu(t)\,dt = d\mu(t)$ is the Riemann--Stieltjes measure generated by
	the function $\mu(t).$ If $\mu(t)$ is absolute continuous, then $\dot\mu(t)$
	is a usual Lebesgue integrable function;\, if $\mu(t)$ is discontinuous
	at a point $t_*\,,$ then $\dot\mu(t)$ contains the Dirac $\delta-$function
	at $t_*\,.$});
hence, one comes to a differential equation of a new, yet
uninvestigated type. Therefore, from the very beginning of studying such
problems, many specialists tried to avoid somehow this difficulty in
order to keep the adjoint equation as an ODE of convenient type.

If the boundary set of the trajectory is a segment, one can differentiate
the state constraint and reduce it to a mixed control-state constraint,
for which the stationarity conditions can be formulated with the usage
of standard objects. The result can be then represented in terms of the
original problem. This way was firstly suggested by R.V.Gamkrelidze in
the classical book \cite{Pont}, earlier than paper \cite{DM65}, but its
realization involves a nontrivial further step:\, one has to obtain
the non-negativity of the measure (the state constraint multiplier), including the
sign of the atoms of measure at junction points,  {which was not completely
	done in \cite{Pont}.}

Thus, for the problems with state constraints there are two forms of optimality
conditions (say, the maximum principle):\, the form of Gamkrelidze and the
form of Dubovitskii--Milyutin. A natural question is how these two forms
are connected. In paper \cite{HSV} and then in \cite{AKP}, it was shown,
by a simple change of the adjoint variable
\footnote{If $\p(t)$ is the adjoint variable in the Dubovitskii--Milyutin form,
	$\Phi(t,x(t))\le 0$ is the state constraint, and a monotone function $\mu(t)$
	generates the corresponding measure, then $\wt\p(t) = \p(t) -\mu(t)\,\Phi'_x(t, x^0(t))$
	is the adjoint variable in the Gamkrelidze form.},
that one can pass from the conditions in the Dubovitskii--Milyutin form
to the conditions in the form of Gamkrelidze, but the possibility of the
inverse passage was not investigated.

In this paper, we consider a special class of problems and reference
trajectories, in which the connection between the non-negativity of the
measure and the minimization of the cost is the most transparent.
In this class, one can completely fulfill Gamkrelidze's idea and prove the
non-negativity of the measure, thus showing that Gamkrelidze's approach allows
to obtain the conditions in Dubovitskii--Milyutin's form.
%Then we show how to extend this conditions on more wide class of problems.
For simplicity, here we consider only necessary conditions of the so-called
{\it extended weak minimality} (i.e., stationarity conditions), leaving the
question about conditions of the strong minimality (the maximum principle)
for further investigations.
\bigskip

%%--------------------------------------------------
\section{Problem Statement}
On a fixed time interval, consider the following optimal control problem
with a state constraint:
\begin{equation}\label{problem_A}
	\mbox{{\textbf{Problem A:}}}\quad
	\begin{cases}
		\begin{aligned}
			\dot z  = f(z, x, u), & \q\; J_A  = J(z(0), z(T),\, x(0), x(T))\to\min, \\
			\dot x  = g(z, x, u), & \q\; \varphi_s\left(u(t)\right) \leq 0,
			\quad s=1,\dots, d(\varphi),\\
			\phantom{r}  x(t) \geq 0.
		\end{aligned}
	\end{cases}
\end{equation}
Here, $z \in \R^n$ and $x\in \R^1$ are state variables, $u\in \R^m$ is a control,
the functions $z(\cdot)$ and $x(\cdot)$ are absolute continuous, $u(\cdot)$ is
measurable and bounded. We will assume that the functions $f,\, g,\, \varphi$
of dimensions $n,\,1,\, \mbox{d}(\varphi),$ respectively, are defined and
continuous on an open subset $\mathcal{Q}\subset\R^{n + 1 +m}$ together
with their first-order partial derivatives w.r.t $z,x,u.$
{(The function $\varphi(u)$ can be formally considered as a
	function of variables $z,x,u$).} Note that the
state constraint is imposed only on the scalar state coordinate $x,$
so it has the simplest form $x\ge 0.$

\begin{definition}
	A triple of functions $w=(z,x,u)$ of the corresponding functional classes
	defined on $[0,T]$ and satisfying equations $\dot z = f(z,x,u),$ $\dot x = g(z,x,u)$
	is called a \textit{process} of problem A. A process is called \textit{admissible}
	if it satisfies all the constraints of the problem.
\end{definition}

%%----------------------------------------------------------
\section{The Reference Trajectory}
%Suppose that
Consider a reference process $w^0= (z^0,x^0,u^0)$  such that the
trajectory $x^0(t)$ touches the state boundary only on a segment $[t^0_1, t^0_2],$
where $0 < t^0_1 < t^0_2 < T.$  In other words, the interval $\Delta:= [0,T]$
is divided into parts $\Delta_1:= [0,t^0_1],$ $\Delta_2:= [t^0_1, t^0_2],$
and $\Delta_3:= [t^0_2, T]$ such that $x^0(t) > 0$ on $[0, t^0_1),$ $x^0(t) = 0$
on $\Delta_2,$ and $x^0(t) > 0$ on $(t^0_2,T].$ In addition, we suppose the control
$u^0$ to be continuous on $\Delta_1\,, \Delta_3$ and Lipschitz continuous on
$\Delta_2$ (for convenience, we assume that the function $u^0$ at time
moments $t^0_1,\, t^0_2$  has both left and right values),
moreover, $\varphi_s(u^0(t))<0$ on $\Delta_2$ for all $s,$ and the following
strict inequalities hold at the moments $t^0_1,\, t^0_2$:
\begin{equation}\label{x12}
	\begin{aligned}
		\dot x^0(t^0_1 - 0)\, &=\; g\left(z^0(t^0_1),\, x^0(t^0_1),\, u^0(t^0_1 - 0) \right) < 0,\\[2pt]
		\dot x^0(t_2^0 + 0)\, &=\; g\left(z^0(t^0_2),\, x^0(t^0_2),\, u^0(t^0_2 + 0) \right) > 0,
	\end{aligned}
\end{equation}
which mean that the landing to the state boundary and the leaving it
occurs with nonzero time derivatives.
We also suppose that $g'_u(z^0(t), x^0(t), u^0(t))\ne 0$ on the boundary arc
$\Delta_2\,,$  {i.e., that the state constraint is of order 1,}
and the gradients $\varphi'_s(u^0(t)),\;\, s \in I(u^0(t)),$ are
positive independent for all $t\in \Delta_1 \cup\Delta_3\,$  {(i.e., their
	nontrivial linear combination with non-negative coefficients cannot vanish)}.
Here $I(u) =\lbrace s\,:\; \varphi_s(u) = 0 \rbrace$
is the set of active indices. \smallskip

For short, we will write the control constraints in the vector form $\varphi(u)\le 0.$
\smallskip

Throughout this paper, we assume that the above assumptions are
satisfied for problem A.

Note that these assumptions are not easily verifiable a priori;\,
however, they are often satisfied in typical real problems.
As any other a priori assumptions, they can be considered, together with
necessary conditions of optimality, as a united collection of conditions
for the search of optimal trajectories.
In the book \cite{Pont}, a less restrictive assumption on the reference
trajectory $x^0(t)$ is imposed:\, it may touch the state boundary not on
one segment, but on a finite number of segments. The reference control
$u^0(t)$ is not assumed in \cite{Pont} to lie in the interior of the set
$\varphi(u)\le 0$ on $\Delta_2\,;$ instead, it is assumed that the gradient
$g'_u(z^0(t), x^0(t), u^0(t))$ together with the active gradients
$\varphi'_s(u^0(t))$ are linearly independent on $\Delta_2\,.$
We do not consider here these more complicated cases in order to avoid
more cumbersome technicalities, which would distract the reader's attention
from the main line of argumentation.

%%---------------------------------------------------------
\section{The Type of Minimum}

We admit not only uniformly small variations of the control, but also
small variations of its discontinuity points. This corresponds to
consideration of the ``extended'' weak minimality. Recall its definition
(see, e.g. \cite{DO}) for a problem of type A.

\begin{definition}
	An admissible process $w^0(t) = (z^0(t), x^0(t), u^0(t))$ provides the
	{\it{extended weak minimumality}} in problem A\, if there exists an $\varepsilon > 0$
	such that, for any Lipschitz continuous surjective mapping
	$\sigma:[0, T]\to [0, T]\,$ satisfying $|\sigma(t) - t| < \varepsilon$
	and $|\dot{\sigma}(t) - 1| < \varepsilon,$ and for any admissible process
	$w(t) = (z(t), x(t), u(t))$ satisfying the conditions
	\begin{equation} \label{exwmin}
		\begin{aligned}
			|z\left(t\right) - z^0\left(\sigma(t)\right)| < \varepsilon,\quad\;
			|x\left(t\right) - x^0\left(\sigma(t)\right)| &< \varepsilon\quad\;\mbox{ for all }t,\\
			|u(t) - u^0\left(\sigma(t)\right)| &< \varepsilon\quad\;\mbox{ for almost all }t,
		\end{aligned}
	\end{equation}
	one has $J(w) \ge J(w^0).$
	
	%(The function $\sigma(t)$ produces a deformation of the time).
\end{definition}

{The conditions on $\sigma$ imply $\sigma(0)= 0$ and $\sigma(T)= T.$
	If we take $\sigma(t) =t,$ then relations (\ref{exwmin}) describe the usual
	uniform closeness between the processes $w^0$ and $w$ both in the state
	and control variables. However, for an arbitrary $\sigma(t),$ relations (\ref{exwmin})
	extend the set of "competing"\, processes, and thus the extended weak minimality
	is stronger than the classical weak minimality. The choice of arbitrary
	$\sigma(t)$ close to $\hat\sigma(t)=t$ corresponds to a variation (deformation)
	of the current time within the interval $[0,T]$ in addition to the usual
	uniformly small variations of $z(t),\,x(t)$ and $u(t)$ for the fixed values
	of $t.$}

{If the control $u^0(t)$ is continuous, the notion of extended weak
	minimality reduces to the usual notion of weak minimality. However, in the
	case of discontinuous $u^0(t),$ the usual small variations of the control
	(corresponding to the weak minimality) leave the points of discontinuity
	of $u^0(t)$ invariable, whereas the extended weak minimality allows for
	small variations of them.}
%\ms

%%--------------------------------------
\section{Passage from Problem A to a Problem with Mixed Control-State Constraints}

Following \cite{DK}, we introduce a new time  {variable} $\tau \in [0,1]$
and consider the initial time variable $t$ on each segment $\Delta_i$
as a new state variable $t_i(\tau)$ subject to equation 
$\dfrac{d t_i}{d \tau} = \rho_i(\tau),$
where the functions $\rho_i(\tau) > 0,$ $i=1,2,3$ are additional controls.

On the segment $[0,1],$ introduce the state variables $r_i(\tau) = z(t_i(\tau)),$
$y_i(\tau) = x(t_i(\tau)),$ and the controls $v_i(\tau) = u(t_i(\tau)).$
Hence, the following equations are satisfied:
$$
\frac{d r_i}{d \tau}  = \rho_i(\tau)\, f(r_i, y_i, v_i),\qquad
\frac{d y_i}{d \tau}  = \rho_i(\tau)\, g(r_i, y_i, v_i), \qquad i=1,2,3.
$$
Thus, we ``replicate'' the variables of the original problem by taking
their reductions to the intervals $\Delta_i$ and considering all of
these reductions as new variables of the new time
\footnote{This natural trick of replication of variables was first
	proposed, probably, in \cite{Den}, and later was also used, may be independently,
	by many authors, e.g. in \cite{Vol-Ost,AgMa,MBKK,ObRo,DK,DK2,LTJW}.}.
In terms of these new variables, we now formulate a new problem
related to our problem A.

Since the original state variables $z,\,x$ are continuous at times $t_1,\,t_2$
(close to $t^0_1,\,t^0_2),$ the new state variables should satisfy the
junction conditions
\beq{ryt01}
\begin{array}{c}
	r_1(1) - r_2(0) =0, \qquad y_1(1) - y_2(0) =0,\qquad t_1(1) - t_2(0) =0, \\[4pt]
	r_2(1) - r_3(0) =0, \qquad y_2(1) - y_3(0) =0,\qquad t_2(1) - t_3(0) =0.
\end{array}
\eeq
Moreover, since the time interval $[0,T]$ is fixed, the variables $t_i$
should satisfy the boundary conditions $\;t_1(0) =0$ and $t_3(1) -T =0.$

Instead of state constraint $y_2(\tau) \geq 0$ on $[0,1],$ we will consider
the following pair of an endpoint and a mixed control-state constraints:
\begin{equation}\label{new-mix}
	y_2(0) \geq 0,\qquad \frac{d y_2}{d \tau} \equiv 0, \quad
	\mbox{ i.e., } \quad g(r_2, y_2, v_2) \equiv 0,
\end{equation}
while the control constraints will be now written in the form
$$
\varphi(v_i(\tau)) \leq 0, \qquad \rho_i >0, \qquad i=1,2,3.
$$

In the new problem, we will consider the ``classical'' weak minimality.
Therefore, we do not need to consider the open constraints $\rho_i >0$
as well as the constraint $\varphi(v_2(\tau)) \leq 0,$ since under our
assumptions the control $v_2^0(\tau)$ lies strictly in its interior.

Thus, we come to the following optimal control problem on the time
interval $\tau \in [0,1]:$
\begin{equation}\label{probl_2_1}
	J_B\; := J\left(r_1(0),\, r_3(1),\, y_1(0),\,  y_3(1)\right)\; \to\min,
\end{equation}
under the following constraints:
\begin{equation}\label{probl_2_2_1}
	\begin{cases}
		\begin{aligned}
			\frac{d r_1}{d \tau} & = \rho_1 f(r_1, y_1, v_1),\quad \; & \q r_1(1) - r_2(0) &= 0,\\
			\frac{d y_1}{d \tau} & = \rho_1 g(r_1, y_1, v_1),\quad \; &\q y_1(1) - y_2(0) &= 0,\\
			\frac{d t_1}{d \tau} & = \rho_1,\qquad  t_1(0) = 0,\; &\q t_1(1) - t_2(0) &= 0,\\
		\end{aligned}
	\end{cases}
\end{equation}

\begin{equation}\label{probl_2_2_2}
	\begin{cases}
		\begin{aligned}
			\frac{d r_2}{d \tau} & = \rho_2 f(r_2, y_2, v_2),\quad &r_2(1) - r_3(0) &= 0,\\
			\frac{d y_2}{d \tau} & = \rho_2 g(r_2, y_2, v_2),\quad &y_2(1) - y_3(0) &= 0,
			\quad & y_2(0) \geq 0,\\
			\frac{d t_2}{d \tau} & = \rho_2,\quad &t_2(1) - t_3(0) &= 0,\\
		\end{aligned}
	\end{cases}
\end{equation}

\begin{equation}\label{probl_2_2_3}
	\begin{cases}
		\begin{aligned}
			\frac{d r_3}{d \tau} & = \rho_3 f(r_3, y_3, v_3),\\
			\frac{d y_3}{d \tau} & = \rho_3 g(r_2, y_2, v_2),\quad \\
			\frac{d t_3}{d \tau} & = \rho_3,\quad &t_3(1) - T = 0,\\
		\end{aligned}
	\end{cases}
\end{equation}

\begin{equation}\label{probl_2_3}
	g(r_2, y_2, v_2) \equiv 0, \qquad \varphi(v_1(\tau)) \leq 0,\qquad
	\varphi(v_3(\tau)) \leq 0.
\end{equation}

This problem will be called {\it{problem B}}.\, Here, $\rho_i, v_i$ are the
controls and $r_i,\, y_i,\, t_i$ the state variables, $i=1,2,3.$
Note that constraints \eqref{new-mix} (included in \eqref{probl_2_2_2} and
\eqref{probl_2_3}) define a smaller class of admissible trajectories than
the state constraint $y_2(\tau) \geq 0$ does, so the new problem is not
equivalent to the initial problem A.\, Later, in Sec. 8, we will also take
into account nonconstant variations of $y_2(\tau),$ i.e., of $x(t)$ on the
boundary interval.\, On the other hand, the new problem does not involve
the state constraints $y_1 \ge 0$ and $y_3 \ge 0,$ so it allows for
a bigger class of admissible trajectories.
\ssk

It is easy to see that, to each admissible process $w = (z,x,u)$ of problem A
with $x(t) = \const$ on an interval $[t_1, t_2],$ one can associate
a (not unique)  admissible process $\gamma = (r_i,y_i,t_i, \rho_i, v_i)$
of problem B  (by choosing, e.g. $\rho_i(\tau) \equiv |\Delta_i|$), and
to each  {admissible} process of problem B one can associate,
{simply by setting $\tau =\tau(t),$} a unique admissible process
of problem A\,  {with $x(t) = \const$ on $[t_1, t_2].$}
\ssk

Let us establish a relation between the extended weak minimality
in problem A and the ``classical'' weak minimality in problem B.

\begin{lemma}
	Let the process $w^0 = (z^0(t), x^0(t), u^0(t))$ with the boundary arc
	$[t_1^0, t_2^0]$ provide the extended weak minimality in problem A.\,
	Then the corresponding process
	$\gamma^0 = (r^0_i(\tau),\, y^0_i(\tau),\, t^0_i(\tau),\, \rho^0_i(\tau),\, v^0_i(\tau),$
	$i=1,2,3)$ provides the weak minimality in problem B.
\end{lemma}

\begin{proof}
	Suppose that the process $\gamma^0$ does not provide the weak minimality in problem B.
	Then, there exists a sequence of uniformly convergent processes
	$\gamma \rightrightarrows \gamma^0$ of  problem B, such that
	$J_B(\gamma) < J_B(\gamma^0).$ According to \eqref{x12}, there exist such
	$\theta < 1$ and $c > 0$  that
	$$
	g\left(r_1^0(\tau), y_1^0(\tau), v_1^0(\tau) \right) \leq - c < 0\;
	\mbox{ on }\;[\theta, 1].
	$$
	Then, for sufficiently far members of the sequence, we get
	$$
	\frac{d y_1}{d\tau} = \rho_1(\tau) g\left(r_1(\tau), y_1(\tau), v_1(\tau)\right) \leq\,
	- \rho_1(\tau) \frac{c}{2}\, < 0\quad\; \mbox{ on }\;[\theta,1].
	$$
	From here with account of $y_1(1) \geq 0$, we get $y_1(\tau) > 0$ on
	$[\theta, 1).$ Consider the segment $[0, \theta].$ Here $y_1^0(\tau) > 0,$
	hence $y_1^0(\tau) \geq b$ for some $b > 0.$ Therefore, $y_1(\tau) \geq b / 2 > 0$
	for sufficiently far members of the sequence. Thus, $y_1(\tau) > 0$ on the
	whole semi-open interval $[0,1).$ Similarly, one    can prove that $y_3(\tau) > 0$
	on the whole semi-open interval $(0,1].$ The inequality $y_2(\tau) \geq 0$
	obviously holds on $[0,1],$ since $y_2 = \const$ and $y_2(0) \geq 0.$
	
	Thus, for the corresponding processes $w = (z,x,u)$ of problem A, we get
	$$
	x(t) > 0\;\mbox{ on }\;[0, t_1)\, \cup\, (t_2, T] \quad \mbox{ and}
	\quad x(t) \geq 0 \; \mbox{ on }\; [t_1,t_2],
	$$
	where $\, t_1 \to t_1^0,\;\; t_2 \to t_2^0\,.$
	The constraints $\varphi(u(t)) \leq 0$ are satisfied in view of
	inequalities $\varphi(v_i(\tau)) \leq 0,$ $i=1,2,3.$
	
	So, the prelimiting processes $w$ are admissible in problem A with the cost
	$J_A(w) = J_B(\gamma) < J_B(\gamma^0) = J_A(w^0),$ a contradiction with
	the extended weak minimality in problem A\, at the process $w^0.$
	\qed
\end{proof}%}

%%--------------------------------------------
\section{Stationarity Conditions for Problem B}\label{section_6}
%To simplify further investigation,
Let us agree to denote the derivatives of $f,\,g$ w.r.t. first, second, and
third arguments as $f'_z,\, g'_z,$ $f'_x,\,g'_x,$ $f'_u,\,g'_u,$ respectively,
no matter
%what symbol  {stands in the place of corresponding arguments of these functions}.
{on which variables these functions depend.}

The three constraints \eqref{probl_2_3} will be treated as mixed control-state
ones. In order to apply the known stationarity conditions, % (see, e.g. \cite{MO,MDO,SIC}),
we have to check whether these constraints are regular along the
reference process $\gamma^0(\tau).$
\ssk

{According to \cite{MO,MDO,SIC}, mixed control-state constraints
	$\Phi_i(t,x,u)\le 0$ and $G_j(t,x,u)=0$ of equality and inequality type
	given by smooth functions on $\R\times \R^n \times \R^r$ are called {\it
		regular}\, at a point $(t,x,u)$ if their gradients w.r.t control
	are positive--linearly independent, which means that there do not exist
	multipliers $\alpha_i \ge 0$ and $\beta_j$ with $\sum \alpha_i + \sum |\beta_j| >0$
	and $\alpha_i\,\Phi_i(t,x,u)=0$ such that
	$$
	\sum \alpha_i\,\Phi'_{iu}(t,x,u)\; + \;\sum \beta_j\,G'_{ju}(t,x,u) =0.
	$$}

Applying this to the constraints \eqref{probl_2_3}, one can
easily see that the gradients w.r.t control $v =(v_1, v_2, v_3)$ of these
constraints are positive--linearly independent along the reference process
{(since they decompose into the gradients w.r.t each component $v_i),$}
hence their gradients w.r.t the ``full'' control vector $(v,\rho)$ are
the more so positive--linearly independent, and thus, the mixed constraints
in problem B\, are regular.
\ssk

Assume the process
$\gamma^0 $:= $(r_i^0(\tau),\, y_i^0(\tau),\, t_i^0(\tau),\, \rho_i^0(\tau),\,
v_i^0(\tau),$ $i=1,2,3)$ provides the weak minimality in problem B.
Then it satisfies the stationarity conditions, which
%in view of the regularity of mixed control-state constraints,
say the following (see, e.g. \cite{MO,MDO,SIC}):\,
there exist multipliers $\alpha_0,\,\alpha_1,$ $\beta_j,\; j=1, ..., 8,$
Lipschitz functions $\psi_{r_i}, \psi_{y_i}, \psi_{t_i},$ $i=1,2,3,$
measurable bounded functions $h_1(\tau),\, h_3(\tau)$ of dimension
$d(\varphi),$ and a measurable bounded scalar function $\sigma(\tau),$
such that the following conditions are satisfied: \\[4pt]
nontriviality condition
\begin{equation}\label{nodegenerativity}
	|\alpha_0| + |\alpha_1| + \sum |\beta_j| + \int_0^1 |h_1(\tau)| d\tau +
	\int_0^1 |\sigma(\tau)| d\tau + \int_0^1 |h_3(\tau)| d\tau > 0,
\end{equation}
non-negativity condition
\begin{equation}\label{nonnegativity}
	\alpha_0 \geq 0, \quad \alpha_1 \geq 0, \quad
	h_1(\tau) \geq 0,\quad h_3(\tau) \geq 0,
\end{equation}
complementary slackness condition
\begin{equation}\label{compslackness}
	\alpha_1\, y_2(0) = 0, \qquad h_1(\tau) \varphi(v^0_1(\tau)) = 0,
	\qquad h_3(\tau) \varphi(v^0_3(\tau))  = 0,
\end{equation}
and such that, in terms of the endpoint Lagrange function
\begin{multline}\label{endpoint_lagrange}
	l = \alpha_0 J\left(r_1(0), r_3(1), y_1(0), y_3(1)\right) +
	\beta_1 t_1(0) +  \beta_2 \left(t_1(1) - t_2(0)\right) + {} \\
	{} +\beta_3 \left(t_2(1) - t_3(0)\right) + \beta_4 \left(t_3(1) -
	T\right) + \beta_5 \left( r_1(1) - r_2(0)\right) +
	\beta_6 \left( r_2(1) - r_3(0)\right) +{} \\
	{} +  \beta_7 \left( y_1(1) - y_2(0)\right) +
	\beta_{8} \left( y_2(1) - y_3(0)\right)  - \alpha_{1} y_2(0)
\end{multline}
and the extended Pontryagin function
\begin{multline}\label{H}
	\overline{\varPi} = \psi_{r_1} \rho_1 f(r_1, y_1, v_1) +
	\psi_{t_1} \rho_1  + \psi_{y_1} \rho_1 g(r_1, y_1, v_1) + {} \\ {}
	+ \psi_{r_2} \rho_2 f(r_2, y_2, v_2) + \psi_{t_2} \rho_2 +
	\psi_{y_2} \rho_2 g(r_2, y_2, v_2) + {} \\ {}
	+ \psi_{r_3} \rho_3 f(r_3, y_3, v_3) + \psi_{t_3} \rho_3  +
	\psi_{y_3} \rho_3 g(r_3,  y_3, v_3) - {} \\ {} -
	\sigma \rho_2 g(r_2, y_2, v_2) - h_1 \varphi\left(v_1 \right) -
	h_3 \varphi\left(v_3 \right), \phantom{tttttt}
\end{multline}
the following conditions are also satisfied: \\[2pt]
adjoint equations and transversality conditions
\begin{equation}\label{adjoint_z}
	\begin{cases}
		\begin{aligned}
			- \frac{d \psi_{r_1}}{d \tau} &=\;
			\rho^0_1 \Big( \psi_{r_1} f'_z(r^0_1, y^0_1, v^0_1) +
			\psi_{y_1} g'_z\left(r^0_1, y^0_1, v^0_1\right)\Big),\\
			- \frac{d \psi_{r_2}}{d \tau} &=\;
			\rho^0_2 \Big( \psi_{r_2} f'_z(r^0_2, y^0_2, v^0_2) +
			(\psi_{y_2} - \sigma)  g'_z\left(r^0_2, y^0_2, v^0_2\right)\Big),\\
			- \frac{d \psi_{r_3}}{d \tau} &=\;
			\rho^0_3 \Big( \psi_{r_3} f'_z(r^0_3, y^0_3, v^0_3) +
			\psi_{y_3}  g'_z\left(r^0_3, y^0_3, v^0_3\right)\Big),\\[4pt]
			\psi_{r_1}(0) &= \alpha_0 J'_{z(0)}\,, \qquad \psi_{r_1}(1) = - \beta_5,\\
			\psi_{r_2}(0) &= - \beta_5, \qquad \quad\;\; \psi_{r_2}(1) = - \beta_6\\
			\psi_{r_3}(0) &= - \beta_6, \qquad \quad\;\;\psi_{r_3}(1) = - \alpha_0 J'_{z(T)}\,,
		\end{aligned}
	\end{cases}
\end{equation}

\begin{equation}\label{adjoint_x}
	\begin{cases}
		\begin{aligned}
			- \frac{d \psi_{y_1}}{d \tau} &=\;
			\rho^0_1 \Big( \psi_{r_1} f'_x(r^0_1, y^0_1, v^0_1) +
			\psi_{y_1} g'_x\left(r^0_1, y^0_1, v^0_1\right)\Big),\\
			- \frac{d \psi_{y_2}}{d \tau} &=\;
			\rho^0_2 \Big( \psi_{r_2} f'_x(r^0_2, y^0_2, v^0_2) +
			(\psi_{y_2} - \sigma) g'_x\left(r^0_2, y^0_2, v^0_2\right)\Big), \\
			- \frac{d \psi_{y_3}}{d \tau} &=\;
			\rho^0_3 \Big( \psi_{r_3} f'_x(r^0_3, y^0_3, v^0_3) +
			\psi_{y_3} g'_x\left(r^0_3, y^0_3, v^0_3\right)\Big) ,\\[4pt]
			\psi_{y_1}(0) &= \alpha_0 J'_{x(0)}\,, \qquad\quad
			\psi_{y_1}(1) = - \beta_7,\\
			\psi_{y_2}(0) &= - \beta_7 - \alpha_{1}, \qquad\;
			\psi_{y_2}(1) = - \beta_{8},\\
			\psi_{y_3}(0) &= - \beta_{8}, \qquad\qquad\;\;
			\psi_{y_3}(1) = - \alpha_0 J'_{x(T)}\,,\\
		\end{aligned}
	\end{cases}
\end{equation}

\begin{equation}\label{adjoint_t}
	\begin{cases}
		\begin{aligned}
			- \frac{d \psi_{t_1}}{d \tau} &= 0, &\quad
			\psi_{t_1}(0) &= \beta_1, &\quad \psi_{t_1}(1) &= - \beta_2\\
			- \frac{d \psi_{t_2}}{d \tau} &= 0, &
			\psi_{t_2}(0) &= - \beta_2, &\psi_{t_2}(1) &= - \beta_3\\
			- \frac{d \psi_{t_3}}{d \tau} &= 0, &
			\psi_{t_3}(0) &= - \beta_3, &\psi_{t_3}(1) &= - \beta_4,\\
		\end{aligned}
	\end{cases}
\end{equation}
stationarity conditions w.r.t\, controls $v_i,$ $i = 1,2,3:$
\begin{equation}\label{stationarity}
	\begin{cases}
		\begin{aligned}
			\overline{\varPi}_{v_1} &= 0,\\
			\overline{\varPi}_{v_2} &= 0,\\
			\overline{\varPi}_{v_3} &= 0,\\
		\end{aligned}
	\end{cases}
	\Leftrightarrow
	\begin{cases}
		\begin{aligned}
			\psi_{r_1}f'_u(r^0_1, y^0_1, v^0_1) + \psi_{y_1} g'_u(r^0_1, y^0_1, v^0_1) & = \frac{h_1 \varphi'_u(v^0_1)}{\rho^0_1}\,,\\
			\psi_{r_2}f'_u(r^0_2, y^0_2, v^0_2) + \psi_{y_2} g'_u(r^0_2, y^0_2, v^0_2) & = \sigma g'_u(r^0_2, y^0_2, v^0_2),\\
			\psi_{r_3}f'_u(r^0_3, y^0_3, v^0_3) + \psi_{y_3} g'_u(r^0_3, y^0_3, v^0_3) & = \frac{h_3 \varphi'_u(v^0_3)}{\rho^0_3}\,,
		\end{aligned}
	\end{cases}
\end{equation}
and stationarity conditions w.r.t\, controls $\rho_i,$ $i = 1,2,3:$
\begin{equation}\label{stationarity_rho}
	\begin{cases}
		\begin{aligned}
			\overline{\varPi}_{\rho_1} &= 0,\\[3pt]
			\overline{\varPi}_{\rho_2} &= 0,\\[3pt]
			\overline{\varPi}_{\rho_3} &= 0,\\[3pt]
		\end{aligned}
	\end{cases}
	\Leftrightarrow
	\begin{cases}
		\begin{aligned}
			\psi_{r_1}f(r^0_1, y^0_1, v^0_1) + \psi_{y_1} g'(r^0_1, y^0_1, v^0_1) + \psi_{t_1} & = 0,\\[3pt]
			\psi_{r_2}f(r^0_2, y^0_2, v^0_2) +(\psi_{y_2} - \sigma) g(r^0_2, y^0_2, v^0_2) + \psi_{t_2} & = 0,\\[3pt]
			\psi_{r_3}f(r^0_3, y^0_3, v^0_3) +\psi_{y_3} g(r^0_3, y^0_3, v^0_3) + \psi_{t_3}& = 0.
		\end{aligned}
	\end{cases}
\end{equation} %\smallskip

\noindent
Here, $J'_{z(0)},\, J'_{z(T)},\, J'_{x(0)},\, J'_{x(T)}$ are the derivatives
of $J(z(0), z(T), x(0), x(T))$ w.r.t the corresponding variables, taken at the
point $(r_1^0(0), r_3^0(1), y_1^0(0), y_3^0(1)).$  \ssk

Note that, since the function $u^0(\tau)$ is Lipschitz continuous on $\Delta_2\,,$
the second equation in \eqref{stationarity} and the nondegeneracy of
$g'_u$ implies that $\sigma(\tau)$ is also Lipschitz continuous.

Fist of all, let us state the following
\begin{lemma}
	$\alpha_0 > 0$ (hence, one can set $\alpha_0 = 1$).
\end{lemma}

\begin{proof}\quad
	Suppose that $\alpha_0 = 0.$ Then by \eqref{adjoint_z}--\eqref{adjoint_x},
	the pair $(\psi_{r_1},\,\psi_{y_1})$ satisfies a linear system of ODEs
	with initial conditions $\psi_{r_1}(0) =0,\; \psi_{y_1}(0) =0,$ whence
	$\psi_{r_1}$ and $\psi_{y_1}$  {identically vanish}.
	Similarly, $\psi_{r_3}$ and $\psi_{y_3}$  {vanish too}, hence
	$\beta_5 = \beta_6 = 0$ and $\beta_7 = \beta_8 = 0.$ %\smallskip
	
	In view of \eqref{stationarity}, we get $h_1(\tau) \equiv h_3(\tau) \equiv 0$
	and $\sigma(\tau) = A \psi_{r_2} + B \psi_{y_2}$ with some Lipschitz continuous
	functions $A(\tau),\, B(\tau);$ moreover,  since $\psi_{r_2}(1) = 0$ and
	$\psi_{y_2}(1) = 0,$ we have $\sigma(1) = 0.$ Thus, in view of
	\eqref{adjoint_z}--\eqref{adjoint_x}, $\psi_{r_2}$ and $\psi_{y_2}$ satisfy
	a system of linear ODEs with zero boundary values at $\tau=1,$ which implies
	that $\psi_{r_2}\equiv 0$ and $\psi_{y_2}\equiv 0.$
	Therefore, $\sigma(\tau) \equiv 0$ and  {by \eqref{adjoint_x}} $\alpha_{1} = 0,$
	then in view of \eqref{stationarity_rho} we get $\psi_{t_1} = \psi_{t_2} =\psi_{t_3} =0,$
	hence $\beta_1 = \beta_2 = \beta_3 = \beta_4 = 0.$ Thus, the whole collection
	of multipliers is trivial, a contradiction with \eqref{nodegenerativity}. \qed
\end{proof}

%%------------------------------------------
\section{Stationarity Conditions in Terms of the Original Problem A}\label{section_7}

Let us rewrite the stationarity conditions from Sec. \ref{section_6} in
terms of the original problem \eqref{problem_A}. To do this, define functions
$m(t)$ and $h(t)$ on the interval $[0,T]$ as follows:
\begin{equation}\label{new_m}
	m(t) := \begin{cases}
		\begin{aligned}
			0 & \; \mbox{ on }\, \Delta_1,\\
			\sigma(\tau(t)) & \; \mbox{ on }\, \Delta_2,\\
			0 & \; \mbox{ on }\, \Delta_3,
		\end{aligned}
	\end{cases} \qquad
	h(t) : = \begin{cases}
		\begin{aligned}
			h_1(\tau(t))  & \;\, \mbox{ on }\,\Delta_1,\\
			0 & \;\, \mbox{ on }\,\Delta_2,\\
			h_3(\tau(t)) & \;\, \mbox{ on }\, \Delta_3.
		\end{aligned}
	\end{cases}
\end{equation}

\noindent
Notice that function $m(t)$ is Lipschitz continuous on the intervals
$\Delta_1,\, \Delta_2,$ $\Delta_3.$ By $\dot{m}(t)$ we will denote its
generalized derivative. Since
$\dfrac{d \psi}{dt} = \dfrac{d\psi}{d\tau} \Big/ \dfrac{dt}{d\tau},$
then, getting back from the new time $\tau$ to the original time $t$ in
equations \eqref{adjoint_z}--\eqref{adjoint_t}, we obtain the following
equations on the whole interval $[0,T]$:
\begin{equation}\label{adjoint_eq_1}
	\begin{cases} \displaystyle
		- \dot \psi_{z} =\; -\frac{1}{\rho^0_i} \frac{d\psi_{r_i}}{d\tau} =\;
		\psi_z f'_z + (\psi_x - m) g'_z\,,  \\[10pt]
		
		\displaystyle -\dot \psi_{x} =\; -\frac{1}{\rho^0_i}
		\frac{d \psi_{y_i}}{d\tau} =\; \psi_z f'_x + (\psi_x - m) g'_x\,, \\[10pt]
		
		\displaystyle  -\dot \psi_{t} =\; -\frac{1}{\rho^0_i}
		\frac{d \psi_{t_i}}{d\tau} =\; 0.
	\end{cases}
\end{equation}

Since the state variables $r_i, y_i, t_i$ of problem B\, are continuously
joined at the corresponding ends of interval $[0,1]$ by the junction conditions
\eqref{ryt01}, the state variables $z(t),\,x(t)$ of problem A are continuous
(and moreover, Lipschitz continuous). By similar arguments, the adjoint
variables $\psi_z,\;\psi_t$ of problem A are also Lipschitz continuous.\,
Consider the function $\psi_x$. \smallskip

Note first that it is Lipschitz continuous on every interval $\Delta_i,$ $i = 1,2,3$.
Rewriting the transversality conditions for $\psi_x$ in terms of problem A,
we get the following junction conditions:
\begin{equation}\label{psi_x_jumps}
	\begin{cases}
		\begin{aligned}
			\psi_x(t_1^0 - 0) &= - \beta_7, \qquad &\psi_x(t_1^0 + 0) &= - \beta_7 - \alpha_{1},\\[2pt]
			\psi_x(t_2^0 - 0) &= - \beta_{8}, \qquad &\psi_x(t_2^0  + 0) &=  - \beta_{8},
		\end{aligned}
	\end{cases}
\end{equation}
i.e., $\psi_x$ is continuous at $t_2^0$ and has the jump
$\Delta\psi_x(t_1^0) = - \alpha_{1} \leq 0$ at the point $t_1.$
At the ends of interval $[0,T],$ it satisfies the transversality conditions
\begin{equation}\label{trat}
	\begin{cases}
		\psi_z(0) = J'_{z(0)}\,, \qquad \psi_z(T) = -J'_{z(T)}\,, \\[2pt]
		\psi_x(0) = J'_{x(0)}\,, \qquad \psi_x(T) = -J'_{x(T)}\,.
	\end{cases}
\end{equation}

If we introduce the extended Pontryagin function for the problem with mixed
control-state constrains
\begin{equation}\label{H_probl_1}
	\overline{K}(z,x,u) = \psi_z f(z, x, u) + \psi_x g(z, x, u) -
	{m} g (z, x, u) - {h} \varphi(u),
\end{equation}
then, in view of \eqref{adjoint_eq_1}, we obtain the fulfilment of adjoint
equations
\begin{equation}\label{adjoint_sys_probl_1}
	- \dot{\psi}_z = \overline{K}'_z(z^0, x^0, u^0),\quad\; -
	\dot{\psi}_x = \overline{K}'_x(z^0, x^0, u^0), \quad\; -
	\dot{\psi}_t = \overline{K}'_t(z^0, x^0, u^0)
\end{equation}
on the interval $[0,T]$ except the points $t_1^0,\,t_2^0\,,$ and the
fulfilment of stationarity condition w.r.t. control $u$ for all $t$:
\begin{equation}\label{stat_conds_probl_1}
	\overline{K}'_u = \psi_z f'_u (z^0, x^0, u^0) +
	(\psi_x - {m})\, g'_u (z^0, x^0, u^0) - {h} \varphi'_u(u^0) = 0.
\end{equation}

Let us now rewrite these conditions in terms of problem A involving a state
constraint. To do this, set $\widetilde{\psi}_x(t) = \psi_x(t) -m(t),$
introduce the Pontryagin function of this problem
$$
%\begin{equation}\label{new_P}
H = \;\psi_z f(z, x, u) + \widetilde{\psi}_x g(z, x, u)
%\end{equation}
$$
and the extended Pontryagin function
\begin{equation}\label{new_H}
	\overline{H} = \;\psi_z f(z, x, u) +
	\widetilde{\psi}_x g\left(z, x, u\right) + \dot m x - h\varphi(u)
\end{equation}
with a multiplier $\dot{m}(t)$ at the state constraint. It is easy to verify
that, along the interval $[0,T]$ except the points $t_1^0,\,t_2^0\,,$
the adjoint equations
\begin{equation}\label{adjoint_sys}
	\dot{\psi}_z =\; - \overline{H}'_z\,,\qquad
	\dot{\widetilde{\psi}}_x =\; - \overline{H}'_x\,,
\end{equation}
and stationarity condition w.r.t. control $\overline{H}'_u = 0$ hold.

The transversality conditions \eqref{trat} are obviously still satisfied.
In view of \eqref{new_m} and \eqref{psi_x_jumps}, the adjoint variable
$\widetilde{\psi}_x$ has the jumps
\begin{equation}\label{pxm}
	\Delta \widetilde\psi_x(t_1^0) = - \alpha_1 - {m}(t_1^0+0),\qquad
	\Delta \widetilde\psi_x(t_2^0)= {m}(t_2^0-0).
\end{equation}

Since $\dot{\psi}_t = 0$, equation \eqref{stationarity_rho} rewritten in
time $t$ turns into
\begin{equation}\label{cons_law}
	\psi_{z} f(z^0, x^0, u^0) + \widetilde{\psi}_{x} g\left(z^0, x^0, u^0\right) + \psi_{t}\; = 0,
\end{equation}
which is equivalent to the ``energy conservation law'' $H(z^0,x^0,u^0) = \const.$
\ssk

Note that we get $x^0 = 0$ on $\Delta_2\,,$ while outside $\Delta_2$ we get
$\dot m \equiv 0,$ i.e., the complementary slackness condition for the
state constraint holds:
\begin{equation}
	\dot{m}(t)\, x^0(t) = 0,\quad \mbox{ i.e., the measure } \quad dm(t)\,x^0(t)=0.
\end{equation}
The definition of ${h}$ and condition \eqref{compslackness} imply that
the complementary slackness condition holds also for the control constraint:
\begin{equation}
	h(t)\, \varphi(u^0(t)) = 0.
\end{equation}

%%--------------------------------------
\section{Non-negativity of Multiplier at the State Constraint}
\label{section_8}

We have obtained stationarity conditions in problem A, in which the measure is
absolute continuous on the interval $\Delta_2$ with density $\dot{m}(t)$
and has the jumps (atoms) $- \alpha_1 - {m}(t_1^0 + 0)$ and ${m}(t_2^0 - 0)$
at the points $t_1^0,\, t_2^0,$ respectively. Our next aim is to define
the sign of its density and jumps.  To this end, we take into account that
we have feasible variations $\bar{x}(t) \ge 0$ on $\Delta_2$ in our disposal.

Consider first any triple $\bar{w}(t) = (\bar{z}(t),\,\bar{x}(t),\,\bar{u}(t))$
satisfying the linearized system in variations along the process $w^0(t)$
on $[0,T]$:
\begin{equation}\label{var_syst}
	\begin{cases}
		\begin{aligned}
			\dot{\bar{z}} &= f'_z \bar{z} + f'_x \bar{x} + f'_u \bar{u}\,, & \\[4pt]
			\dot{\bar{x}} &= g'_z \bar{z} + g'_x \bar{x} + g'_u \bar{u}\,. &
		\end{aligned}
	\end{cases}
\end{equation}

The main technical formula to use is defined by the following

\begin{lemma}\label{lemma_variation}
	Let be given Lipschitz continuous functions $\psi_z(t),\, z(t),\, x(t)$
	and measurable bounded functions $h(t),\, u(t)$ on an interval $[0,T].$
	Let be also given functions $\psi_x(t),\, m(t)$ Lipschitz continuous on
	intervals $\Delta_1 =[0,t_1],$ $\Delta_2 =[t_1,t_2],$ $\Delta_3 =[t_2,T]$
	with possible jumps at the points $t_1,\,t_2,$ where $0< t_1 <t_2 <T,$
	such that the following relations hold on every above interval:
	\begin{equation}\label{lemma_eq_2}
		\begin{aligned}
			\dot \psi_z = - \psi_z f'_z - (\psi_x  - m )\, g'_z,
			\qquad  \dot \psi_x = - \psi_z f'_x - (\psi_x - m )\,g'_x, \\[4pt]
			\psi_z f'_u\, + (\psi_x - m )\, g'_u - h \varphi'_u = 0.
		\end{aligned}
	\end{equation}
	Then any solution $\bar{w} = (\bar{z},\bar{x},\bar{u})$ of system
	\eqref{var_syst} on $[0, T]$ satisfies the following equality:
	\begin{multline}\label{lemma_res}
		\psi_z(T) \overline{z}(T) + \psi_x(T) \overline{x}(T) -
		\psi_z(0) \overline{z}(0) - \psi_x(0) \overline{x}(0)\; =
		\int_0^{t_1} m \dot{\bar{x}}\,dt\, + \int_{t_2}^{T} m \dot{\bar{x}}\,dt + {} \\
		+ \bigl(\Delta \psi_x(t_1)- m(t_1+0)\bigr)\, \bar{x}(t_1) +
		\bigl(\Delta \psi_x(t_2)\; +\;  m(t_2-0)\bigr)\, \bar{x}(t_2)\,-  \\ -
		\int_{t_1}^{t_2} \dot{m} \bar{x}\, dt\, + \int_0^T h \varphi'_u \bar{u}\, dt,
		\phantom{rr}
	\end{multline}
	where $\Delta \psi_x(t_{i})$ are the jumps of $\psi_x$ at the points
	$t_1,\, t_2\,.$
\end{lemma}

\begin{proof}
	In view of \eqref{lemma_eq_2}, we have, on every interval $\Delta_i$:
	\begin{multline*}
		\frac{d}{dt}\left( \psi_z \bar{z} + \psi_x \bar{x} \right) = \;
		\left(- \psi_z f'_z - (\psi_x - m) g'_z \right) \bar{z}\; +\,
		\psi_z (f'_z \bar{z} + f'_x \bar{x} + f'_u \bar{u}) + {} \\
		{} + \left(- \psi_z f'_x - (\psi_x - m) g'_x \right) \bar{x}\, +\,
		\psi_z (g'_z \bar{z} + g'_x \bar{x} + g'_u \bar{u})\; = {} \\ {} =\;
		m g'_z \bar{z} + m g'_x \bar{x}  + m g'_u \bar{u} + h \varphi'_u \bar{u}\;
		=\; m \dot{\overline{x}}\, +\, h \varphi'_u \overline{u}.
	\end{multline*}
	Integrating this equality on the whole interval $[0,T]$ (on $\Delta_2\,,$ we integrate
	$m\dot{\bar{x}}$ by parts)\, and taking into account possible jumps of $\psi_x$
	at the points $t_1,\,t_2,$ we get that the left hand part of \eqref{lemma_res} is equal to
	\begin{multline*}
		\int_0^T d\,(\psi_z \bar{z} + \psi_x \bar{x})\; =
		\int_0^{t_1} m \dot{\bar{x}}\, dt\, + \int_{t_2}^T m \dot{\bar{x}}\, dt\, +\,
		m\,\bar{x}\Bigr|_{t_1}^{t_2}\, -  \int_{t_1}^{t_2} \dot m\, \bar{x}\, dt\, +\; \\
		+\, \Delta \psi_x(t_1)\, \bar{x}(t_1)\, +\, \Delta \psi_x(t_2)\, \bar{x}(t_2)\, +
		\int_0^T h\, \varphi'_u \bar{u}\, dt,
	\end{multline*}
	which implies the required equality \eqref{lemma_res}. \qed
\end{proof}

Now, we introduce variations of some special type.

\begin{lemma}\label{lemma_variation_1}
	For any Lipschitz continuous function $\varkappa(t)$ defined on the interval
	$\Delta_2 =[t_1^0,\,t_2^0],$    there exists a solution $(\bar{z}(t),\bar{x}(t),\bar u(t))$
	of system \eqref{var_syst} on $\Delta_2$ such that $\bar{x}(t) = \varkappa(t).$
\end{lemma}

\begin{proof}
	Let us set $\bar{x}(t) = \varkappa(t),\,$ $\bar{u}(t) = v(t)\, g'_u\,,$
	where $v(t)$ is a scalar function to be found. Since $g'_u(z^0,x^0,u^0)\ne 0,$
	from the second equation of system \eqref{var_syst} we obtain
	$
	v(t) =\, \big(\dot{\varkappa} - g'_z \bar{z} - g'_x \varkappa\big) /\, |g'_u|^2\,.
	$
	Substituting the corresponding $\bar{u}(t)$ into the first equation of
	system \eqref{var_syst}, we come to the following nonhomogeneous equation
	with respect to $\bar{z}:$
	$$
	\dot{\bar{z}}\, =\, f'_z \bar{z}\, - \dfrac{\langle g'_z,\bar{z}\rangle}{|g'_u|^2}\,  f'_u g'_u\, +\,
	\left(f'_x - \dfrac{g'_x}{|g'_u|^2}\, f'_u g'_u \right) \varkappa\, +\,
	f'_u \frac{\dot{\varkappa}}{|g'_u|^2}\, g'_u\,.
	$$
	Setting for definiteness $\bar{z}(t_1) = 0,$ we get the solution of
	this equation, and then define $v(t)$ and $\bar u(t).$ \qed
\end{proof}

Consider now any $\varkappa(t) > 0$ on $\Delta_2 =[t_1^0,\, t_2^0].$
By Lemma \ref{lemma_variation_1}, the system \eqref{var_syst} has a solution
$\bar w(t)= (\bar{z}(t), \bar x(t), \bar u(t))$ on $\Delta_2$ with
$\bar{x}(t) = \varkappa(t).$ To construct the corresponding process, which
will be compared with the optimal one $w^0,$ we have to go back to the
original nonlinear system $\dot z = f(z,x,u),$ \, $\dot x = g(z,x,u).$
Note that \eqref{var_syst} is the variational system for the latter one.
According to the main property of variational equation, for any $\e > 0$
there exists a correction $\tilde w_{\e} = (\tilde{z}_{\e}, \tilde{x}_e, \tilde{u}_e)$
with $||\tilde w_\e||_\infty \leq o(\e)$ as $\e \to 0+$ such that the
triple $w_\varepsilon = w^0 + \varepsilon \bar w + \tilde w_\varepsilon$
satisfies the original system on $\Delta_2.$ It is easy to verify that
this triple satisfies also conditions $x_\e(t)>0$ and $\varphi(u_\e)<0$
on $\Delta_2\,.$
%\smallskip

Now, let us extend this triple, defined only on $\Delta_2\,,$ to a process
defined on the whole interval $[0,T].$ To do this, on $\Delta_1 =[0, t_1^0]$
we set $u_\e = u^0$ (i.e., $\bar{u} =0$) and solve the nonlinear system with
initial conditions $z_\e(t_1^0),\; x_\e(t_1^0).$ \mbox{On $\Delta_3\,,$}
we again set $u_\e = u^0\; (\bar{u} =0)$ and solve the nonlinear system with
the initial conditions $z_\e(t_2^0),\; x_\e(t_2^0).$ Thus, we get a process
$w_\e = (z_\e,\, x_\e,\, u_\e)$ on the whole interval $[0,T]$ that by definition
satisfies the constraint $\varphi(u_\e) \leq 0.\;$

Note that $\dfrac {d w_\e(t)}{d\e} = \bar{w}(t) = (\bar{z},\bar{x},\bar{u}),$
where $\bar{u}=0$ on $\Delta_1 \cup \Delta_3\,$ and $\bar u$ on $\Delta_2$
is the above function from Lemma \ref{lemma_variation_1}, satisfies the
linear system \eqref{var_syst} on $[0,T].$ In particular, the  pair
$\left(\bar{z} = \dfrac{d z_{\e}}{d \e},\;\,\bar{x} = \dfrac{d x_{\e}}{d \e}\right)$
satisfies on $\Delta_1 \cup \Delta_3$ the system of linear equations in variations
\beq{syst-zx}
\dot{\bar{z}} = f'_z \bar{z} + f'_x \bar{x}, \qquad
\dot{\bar{x}} = g'_z \bar{z} + g'_x \bar{x}\,.
\eeq

\begin{lemma}\label{xpol}\,
	$x_\varepsilon(t)>0$ on $\Delta_1 \cup \Delta_3$ for small $\varepsilon>0,$
	except the points $t_1^0,\, t_2^0\,.$
\end{lemma}

\begin{proof} Define $\zeta(t) = \bar z(t)$ on $\Delta_2$ and consider the interval
	$\Delta_3.$ On this interval, the pair $(z_{\e}, x_{\e})$ satisfies the
	same nonlinear system as the pair $(z^0,x^0),$ but with the corrected initial
	conditions $z_{\e}(t_2^0),\, x_{\e}(t_2^0).$ Then, the pair $(\bar z, \bar x)$
	satisfies the linear system \eqref{syst-zx} with initial conditions
	$\bar{z}(t_2^0) = \zeta(t_2^0),$ $\bar{x}(t_2^0) = \varkappa(t_2^0).$
	Since $c:= \varkappa(t_2^0) > 0,$ there exists such $\delta > 0$ that
	$\bar{x}(t) \ge c/2$ on $[t_2^0,\, t_2^0 + \delta].$ Then
	$x_{\e}(t) \ge \e c/ 3$ on this interval for small enough $\e > 0.$
	\smallskip
	
	Since $x^0(t) \ge \const > 0$ on $[t_2^0 + \delta, T],$ we get
	$x_{\varepsilon}(t) > 0$ for small $\varepsilon > 0.$
	Thus, $x_{\varepsilon}(t) > 0$ on the whole $\Delta_3 \setminus \{t_2^0\}.$
	The interval $\Delta_1$ is considered similarly. \qed
\end{proof}

Thus, the constructed process $w_\varepsilon$ satisfies all the constraints
of problem \eqref{problem_A} and, since  the process $w^0$
provides the weak minimality, we have
\begin{equation}\label{DJ}
	\frac d{d\varepsilon}\, J(\tilde w_\varepsilon)\Bigr|_{\varepsilon=0}\; =\; J'(w^0)\, \bar{w}\; \ge\; 0.
\end{equation}

Let us apply Lemma \ref{lemma_variation} to the constructed triple $(\bar{z}, \bar{x}, \bar{u})$
and functions $\psi_x,\, m,\, h$ defined in \eqref{new_m}--\eqref{trat}.\,
Since $m=0$ on $\Delta_1 \cup \Delta_3\,,$ the first two integrals in
\eqref{lemma_res} disappear, and since $h=0$ on $\Delta_2$ and  $\bar{u}=0$
on $\Delta_1 \cup \Delta_3\,,$ the last integral disappears too.
According to transversality conditions \eqref{trat}, the left hand part of relation
\eqref{lemma_res} is exactly
$$
-\Bigl(J_{z(0)}\bar z(0) + J_{z(T)}\bar z(T)+ J_{x(0)}\bar x(0) + J_{x(T)}\bar x(T)\Bigr)
\;=\, -J'(w^0)\,\bar{w},
$$
hence  \vadjust{\kern-8pt}
\begin{multline}\label{dJ}
	J'(w^0)\,\bar{w}\; = {} \\ {} =
	\Bigl(-\Delta \psi_x(t_1^0) + m(t_1^0 + 0) \Bigr)\, \bar{x}(t_1^0)\, -\,
	\Bigl(\Delta \psi_x(t_2^0) + m(t_2^0 - 0) \Bigr)\, \bar{x}(t_2^0)\, +
	{} \\ {} + \int_{\Delta_2} \dot{m}\, \bar{x}\, dt \; \ge\; 0. \phantom{rrr}
\end{multline}

This inequality holds for any Lipschitz continuous function $\bar{x}(t) = \varkappa(t) > 0$
on $\Delta_2.$ Now, take any $\varkappa(t) \ge 0$ on $\Delta_2\,.$
Approximating it uniformly by functions $\varkappa(t) > 0$ and passing to
the limit in \eqref{dJ}, we obtain that inequality \eqref{dJ} holds for
any Lipschitz continuous function $\varkappa(t) \ge 0$ on $\Delta_2\,.$
Considering only $\varkappa(t)$ with zero values at the endpoints of $\Delta_2,$
we get $\int_{\Delta_2} \dot{m}\, \bar{x}\, dt  \ge 0,$ which implies
$\dot m(t) \ge 0$ almost everywhere on $\Delta_2\,.$ \smallskip

Consider now functions $\varkappa(t) \ge 0$ that vanish on
$[t_1^0 +\delta,\,t_2^0]$ for a small $\delta>0$ and satisfy $\varkappa(t) \leq 1$
with $\varkappa(t_1^0)=1.$ If $\delta \to 0+\,,$ the integral in the right
hand side of \eqref{dJ} tends to zero, thus $-\Delta \psi_x(t_1^0) + m(t_1^0 + 0) \ge 0.$
In view of equality $\Delta\psi_x(t_1^0) = - \alpha_{1},$ we get
$\alpha_1 + m(t_1^0 + 0) \ge 0.$ Similarly, we  get $ - m(t_2^0 - 0) \ge 0$
in view of continuity of $\psi_x$ at the point $t_2^0\,.$ \smallskip

Thus, we have proved the following

\begin{lemma}\label{lemma_min}
	Let the process $w^0$ provide the extended weak minimality in problem~A.
	Then $\dot{m}(t) \ge 0$ on $\Delta_2$ (i.e., $m(t)$ decreases on $\Delta_2);$
	moreover, $\alpha_1 + m(t_1^0 + 0) \ge 0$ and  $- m(t_2^0 - 0) \ge 0.$
\end{lemma}

Let us get back to the function $\widetilde{\psi}_x = \psi_x -m$ having the jumps \eqref{pxm}. \\
To ``equalize'' these jumps, we introduce the function
\begin{equation}\label{state_multiplier}
	\mu (t) =
	\begin{aligned}
		\begin{cases}
			\begin{aligned}
				-\alpha_1 - m(t_1^0 + 0)  & \;\; \mbox{ on }\, \Delta_1,\\[2pt]
				m(t) - m(t_1^0 + 0) & \;\; \mbox{ on }\, \Delta_2,\\[2pt]
				- m(t_1^0 + 0) & \;\; \mbox{ on }\, \Delta_3.
			\end{aligned}
		\end{cases}
	\end{aligned}
\end{equation}
Then, according to Lemma \ref{lemma_min},
$$
\Delta \mu(t_1^0) = \alpha_1 + m(t_1^0 + 0) \ge 0, \qquad
\Delta \mu(t_2^0) = - m(t_2^0 - 0) \ge 0,
$$
and $\dot \mu(t) \ge 0$ for $t \ne t_1^0,\;\, t \neq t_2^0.$ The jumps of
adjoint variable $\widetilde{\psi}_x$ at junction points have now a ``symmetric''
form: $\;\Delta \widetilde{\psi}_x(t_i^0) =\, - \Delta \mu(t_i^0),$ $i=1,2.$
The adjoint equation for $\widetilde{\psi}_x$ (see \eqref{adjoint_sys} now
looks as follows:
$$
\dot{\widetilde{\psi}}_x = -\psi_z f'_x(z^0, x^0, y^0) +
\widetilde{\psi}_x g'_x(z^0, x^0, u^0)) - \dot \mu(t), \qq t\in [0,T],
$$
where $\dot\mu$ is the derivative in the sense of generalized functions.
This equation  {should} be regarded as an equality between measures:
$$
d \widetilde{\psi}_x = - \left(\psi_z f'_x(z^0, x^0, y^0) +
\widetilde{\psi}_x g'_x(z^0, x^0, u^0)\right) dt - d \mu(t), \q t\in [0,T].
$$

%%--------------------------------------------------------------
\section{The Final Result}\label{section_9}
%Thus, we proved the following
{We now summarize our findings:}

\begin{theorem}\label{stat}
	Let $w^0(t) = \left(z^0(t), x^0(t), u^0(t)\right)$ be an admissible process
	in problem A
	such that $x(t) = 0$ on $\Delta_2 = [t^0_1,t^0_2],$ $x(t) > 0$ on
	$[0,T]\setminus \Delta_2,$ $\varphi_i(u^0(t)) < 0$ on $\Delta_2\,,$
	assumption \eqref{x12} holds, and let this process provide the extended weak minimality.
	Then there exist a Lipshitz continuous function $\psi_z(t),$ a constant $c,$
	functions $\widetilde{\psi}_x(t)$ and $\mu(t)$ Lipschitz continuous on each interval
	$\Delta_i,\;\, i=1,2,3,$ with possible jumps at $t^0_1,\,t^0_2\,,$ and a
	measurable bounded function $h(t),$ which generate the Pontryagin function
	$$
	H(z, x, u) =\; \psi_z f(z, x) + \widetilde{\psi}_x\, g(z, x, u),
	$$
	and the extended Pontryagin function
	$$
	\overline{H} =\; \psi_z f(z, x, u) +
	\widetilde{\psi}_x\, g(z, x, u) + \dot\mu x - h \varphi(u),
	$$
	such that the following conditions hold:
	
	\begin{enumerate}[(a)]
		\item non-negativity conditions
		\begin{equation}\label{nonnegativity_cond}
			\dot{\mu}(t)\;\ge 0\;\mbox{ a.e.\, on }\, \Delta_2,
			\quad h(t) \ge 0 \; \mbox{ a.e.\, on }\, [0,T],
		\end{equation}
		
		\item complementary slackness
		\begin{equation}\label{comp_slack_cond}
			d\mu(t)x^0(t) = 0, \quad\; h(t) \varphi(u^0(t)) \;
			\mbox{ a.e.\, on } [0,T],
		\end{equation}
		
		\item  adjoint equations
		\begin{equation}\label{adjoint_eq_5}
			\begin{cases}
				\begin{aligned}
					\dot \psi_{z} & =\; - \psi_z f'_z(z^0, x^0, u^0) - \widetilde{\psi}_x g'_z(z^0, x^0, u^0),\\
					\dot {\widetilde{\psi}}_x & =\; - \psi_z f'_x(z^0, x^0, u^0) - \widetilde{\psi}_x g'_x(z^0, x^0, u^0) - \dot\mu,
				\end{aligned}
			\end{cases}
		\end{equation}
		
		\item transversality conditions
		\begin{equation}\label{trans_cond_5}
			\begin{cases}
				\begin{aligned}
					\psi_z(0) &= J'_{z(0)}, &\qquad \psi_z(T) &= - J'_{z(T)},\\
					\psi_x(0) &= J'_{x(0)}, &\qquad \psi_x(T) &= - J'_{z(T)},\\
				\end{aligned}
			\end{cases}
		\end{equation}
		
		\item  jumps conditions for the adjoint variable $\widetilde{\psi}_x$
		\begin{equation}\label{jump_cond}
			\Delta\widetilde{\psi}_x(t_1^0) = - \Delta\mu(t_1^0) \leq 0,
			\qquad \Delta\widetilde{\psi}_x(t_2^0) = - \Delta\mu(t_2^0) \leq 0,
		\end{equation}
		
		\item the energy conservation law
		\begin{equation}\label{energy_conservation}
			H(z^0(t),x^0(t),u^0(t)) = c,
		\end{equation}
		
		\item stationarity condition w.r.t. control
		\begin{equation}\label{stat_cond}
			\overline{H}'_u\left(z^0(t),x^0(t),u^0(t)\right) = 0\quad
			\mbox{ a.e. on }\; [0,T].
		\end{equation}
	\end{enumerate}
\end{theorem}

\begin{remark}\, Note again that theorem \ref{stat} is not new;\,
	in fact, it is the stationarity conditions in the Dubovitskii--Milyutin's
	form with some refinements for our specific problem A. The novelty is
	only in the way of obtaining this result.
\end{remark}

\begin{remark}\, If the functions $\varphi_s(u),\; s=1, \ldots, d(\f)$
	are convex  and the function $H(z^0, x^0, u)$ turns out to be concave in $u,$ then,
	as is known, stationarity condition \eqref{stat_cond} is equivalent
	to the maximality condition over the set
	$U= \{ u\;|\;\varphi_s(u) \le 0,$  $s=1, \ldots, d(\f) \}:$
	\begin{equation}\label{max_cond}
		H\left(z^0(t),x^0(t),u^0(t)\right)\, =\;
		\max_{v\in U}H(z^0(t),x^0(t),v)\;\; \mbox{ for a.a. }\;t,
	\end{equation}
	i.e., the necessary conditions for the extended weak minimality and for the
	strong minimality are equivalent. However, if the cost $J$ is not convex,
	then neither strong, nor even weak minimality can be guaranteed.
\end{remark}

\begin{remark}
	Note that, in the proof of theorem \ref{stat}, the variation of the reference
	process are made in two stages, not in one, as usual.
	First, we use not the whole class of possible variations, but only those
	for which $\bar x= \mbox{const}$ on the boundary interval $\Delta_2\,.$
	In the second stage, we consider the stationarity conditions obtained for
	this reduced class, and substitute to them the ``remaining'' variations
	$\bar x \ge 0$ concentrated inside the boundary interval $\Delta_2$ and
	near its endpoints, which makes it possible to specify these conditions.
	This approach might be feasible not only for the given class of problems,
	but also for some other problems (see, e.g. Sec. \ref{section_12}--\ref{section_14}
	below).
\end{remark}

%%---------------------------------------------------
\section{On the Jumps of Measure -- the Multiplier at the State Constraint }\label{section_10}
Of special interest is the question, in which cases the adjoint variable
$\widetilde{\psi}_x(t)$ and the function $\mu(t)$ generating the measure do not
have jumps at junction points?\, Studies show (see, e.g. the book
\cite[\S{6}]{ADMC} or papers \cite{HSV,M77,LTJW,BdV,Pin-Sch,A-2016})
that in case of strong (or at least Pontryagin type \cite{MO,MDO}) minimality,
the adjoint variable and measure do not have jumps under condition \eqref{x12}.
However, this result is not, in general, valid in the case of extended weak
minimality (the reason is that one cannot rely upon the maximality of Pontryagin
function w.r.t. $u,$ having in disposal only the stationarity of the extended
Pontryagin function). Here, we specify a class of problems where the adjoint
variable and measure have no jumps, and also present an example where
the adjoint variable and measure corresponding to a stationary (but not
optimal) trajectory do have nonzero jumps at junction points.

%%------------------------------
\subsection{On the Absence of Atoms of Measure}
Consider the case when the dynamics of the ``free'' state variable $z$
does not depend on $u:\;$ $\dot z  = f(z,x).$
Applying \eqref{stat_cond} to the interval $\Delta_2\,,$ we get
$$
\overline{H}'_u =\; \widetilde{\psi}_x\, g'_u(z^0, x^0, u^0) = 0,
$$
whence, in view of assumption $g'_u(z^0, x^0, u^0) \ne 0,$ obtain
$\widetilde{\psi}_x \equiv 0$ on $\Delta_2\,.$ Thus,
$ \widetilde{\psi}_x (t_1-0) + \Delta\widetilde{\psi}_x (t_1) =
\widetilde{\psi}_x (t_1+0) =0,
$
and so $ \Delta\widetilde{\psi}_x (t_1) = -\widetilde{\psi}_x (t_1-0).$
\smallskip

According to the energy conservation law \eqref{energy_conservation}, the jump
of the so-called switching function (the $u$-dependent term of Pontryagin function)
at the point $t_{1}$ is zero:
%\begin{equation}\label{jump}
$$
0 = \Delta \left(\widetilde{\psi}_x\, g\right)(t_1) =\,
\widetilde{\psi}_x(t_1+0)\, g(t_1 + 0) -
\widetilde{\psi}_x (t_1 - 0)\, g(t_1-0) = \Delta\widetilde{\psi}_x (t_1)\, g(t_1-0),
%\end{equation}
$$
where $g(t_1 \pm 0) := g\left(z(t_1), x(t_1), u(t_1 \pm 0)\right) \ne 0$
according to \eqref{x12}, and therefore $\Delta \widetilde{\psi}_x (t_1) = 0.$
One can similarly show that $\Delta \widetilde{\psi}_x (t_2) = 0$ either.
\smallskip

Thus, in the considered case, the measure has no atoms, and the adjoint variables
are continuous. In the general case, the question of presence or absence of atoms
is open. We leave it for further research. \smallskip

%%---------------------
\subsection{An Example Where the Measure Has Atoms}

Consider the following problem:
\begin{equation}\label{primer}
	\begin{cases}
		\begin{aligned}
			\dot z & = f(u),\qquad  u^2 - 1 \leq 0, \\
			\dot x & = u, \qquad \qquad x \ge 0, \\
			J& = z(0) - z(3) +\, a\left(x(0) + x(3)\right)\to \min,
		\end{aligned}
	\end{cases}
\end{equation}
where $z,\, x,\,u \in \R,$ and a parameter $a>0$ is arbitrary.
Let $\Delta_1 = [0,1],$ $\Delta_2 = [1,2],$ $\Delta_3 = [2,3].$
Consider a trajectory generated by the control $u=(-1,0,1)$ on
$\Delta_1, \Delta_2, \Delta_3,$ for which $x = (1-t,\; 0,\; t-2)$
on $\Delta_1, \Delta_2, \Delta_3,$ respectively. The value of $z$
is defined up to an additive constant, which does not matter.

Let this trajectory satisfy the stationarity conditions of Theorem \ref{stat},
i.e., let there exist Lipschitz continuous function $\psi_z,$ Lipschitz
continuous on $\Delta_1,\Delta_2, \Delta_3$ functions $\psi_x$ and $\mu$ with
possible jumps at $t_1=1$ and $t_2=2,$ a constant $c,$ and a measurable bounded
function $h,$ which generate the Pontryagin function $H = \psi_z f(u) + \psi_x u$
and the extended Pontryagin function
$$
\overline{H} = \psi_z f (u) + \psi_x u +\, \dot{\mu}\, x - h(u^2 -1),
$$
such that the following condition hold:
\begin{enumerate}[(a)]
	\item adjoint equations
	\begin{equation}\label{ex_adjoint}
		\dot{\psi}_z = 0, \qquad \dot{\psi}_x = - \dot{\mu},
	\end{equation}
	\item transversality conditions
	\begin{equation}\label{ex_transvers}
		\begin{cases}
			\begin{aligned}
				\psi_z(0) &= J'_{z(0)} =1, \qquad &\psi_z(3) &= -J'_{z(T)} =1,\\[4pt]
				\psi_x(0) &= J'_{x(0)} =a , \qquad &\psi_x(3) &= -J'_{x(T)} = -a ,
			\end{aligned}
		\end{cases}
	\end{equation}
	\item complementary slackness conditions
	\begin{equation}\label{ex_slackness}
		\dot{\mu}\, x = 0, \qquad h\,(u^2 - 1) = 0,
	\end{equation}
	\item stationarity conditions w.r.t. control
	$$
	\overline{H}'_u = 0\; \Longleftrightarrow \;
	\begin{cases}
	\begin{aligned}
	\psi_z f'(-1) + \psi_x &= 2 h u, &\; \mbox{on }\; \Delta_1,\\[3pt]
	\psi_z f'(0) + \psi_x &= 0, &\;  \mbox{on }\; \Delta_2,\\[3pt]
	\psi_z f'(1) + \psi_x &= 2 h u, &\,  \mbox{on }\; \Delta_3\,,\\
	\end{aligned}
	\end{cases}
	$$
	that imply the adjoint variable to be as follows
	\begin{equation}\label{psi_x}
		\psi_x=
		\begin{cases}
			\begin{aligned}
				-2 h - \psi_z f'(-1),\quad  \mbox{on }\; \Delta_1,\\[3pt]
				- \psi_z f'(0),\quad  \mbox{on }\; \Delta_2,\\[3pt]
				2 h - \psi_z f'(1),\quad  \mbox{on }\; \Delta_3,\\
			\end{aligned}
		\end{cases}
	\end{equation}
	\item and the energy conservation law
	\begin{equation}
		H = c\; \Longleftrightarrow \;
		\begin{cases}
			\begin{aligned}
				\psi_z f(-1) + 2 h - \psi_z f'(-1) &= c, &\;  \mbox{on }\; \Delta_1,\\[3pt]
				\psi_z f(0) &= c, &\; \mbox{on }\; \Delta_2,\\[3pt]
				\psi_z f(1) + 2 h - \psi_z f'(1) &= c, &\;  \mbox{on }\; \Delta_3\,.
			\end{aligned}
		\end{cases}
	\end{equation}
\end{enumerate}
From \eqref{ex_adjoint}--\eqref{ex_transvers} it follows that $\psi_z \equiv 1.$
Set  $ h = (1,\,0,\,1)$ on $\Delta_1,\Delta_2, \Delta_3\,.$ The complementary
slackness conditions are then obviously hold. Thus, according to \eqref{psi_x}, we get
\begin{equation}\label{[ex_psi_x]}
	\psi_x =
	\begin{cases}
		\begin{aligned}
			-2 - f'(-1), \quad & \; \mbox{on }\;  \Delta_1,\\[2pt]
			- f'(0), \quad & \; \mbox{on }\; \Delta_2,\\[2pt]
			2 - f'(1), \quad & \; \mbox{on }\; \Delta_3,
		\end{aligned}
	\end{cases}
\end{equation}
while the energy conservation law reads as follows:
\begin{equation}\label{ex_stat}
	\begin{cases}
		\begin{aligned}
			f(0) &= f(-1) + 2 - f'(-1),\\[3pt]
			f(0) &= f(1) + 2 - f'(1).
		\end{aligned}
	\end{cases}
\end{equation}
Conditions \eqref{ex_stat} are definitely satisfied if, e.g., $f$ is such that
\begin{equation}\label{ex_f}
	\begin{cases}
		\begin{aligned}
			f(-1) &= a , & \q f'(-1) &= - 2 - a, \\[2pt]
			f(0) &= 0, & f'(0) &= 0,\\[2pt]
			f(1) &= a ,  & f'(1) &= 2 + a.
		\end{aligned}
	\end{cases}
\end{equation}
Then, the transversality conditions \eqref{ex_transvers} hold too,
and the jumps of $\psi_x$ at the points $1$ and $2$ are
\begin{equation*}
	\begin{aligned}
		\Delta \psi_z(1) &=\; 2+ \left(f'_u(-1) - f'_u(0)\right) =\; - a  < 0,\\[4pt]
		\Delta \psi_z(2) &=\; 2+ \left(f'_u(0) - f_u'(1)\right) =\; - a  < 0.
	\end{aligned}
\end{equation*}

Now, it remains to find a smooth function $f$ satisfying conditions \eqref{ex_f}.\\
To this purpose one can use, e.g. the following polynomial:
$$
f(u) = \Bigl(1 - \frac{a }{2}\Bigr)u^4 +
\Bigl(\frac{3a }{2} -1\Bigr) u^2.
$$
Thus, we get a stationary trajectory for which the adjoint variable $\psi_x$
has jumps $-a $ at the points $t = 1,\,2.$ (Choosing a corresponding
$f,$ one can make these jumps not equal.) \smallskip

Note that here the Pontryagin function $H$ is not concave in $u$ (on the
contrary, it is convex), so the stationarity conditions w.r.t. control
$\overline{H}'_u = 0$ does not ensure the maximum of $H,$ i.e., the reference
trajectory does not satisfy the maximum principle, and hence, it is just
stationary but does not provide the strong minimality.

Thus, the stationarity conditions do not guarantee the absence of atoms,
while, according to \cite{HSV,M77,ADMC,LTJW,BdV,Pin-Sch,A-2016}),
the maximum principle does. If a trajectory is not just stationary, but
provides the strong (or at least Pontryagin type) minimality, then it
satisfies the maximum principle, and therefore, the corresponding measure
cannot have atoms.
\bigskip

%%------------------------------------------
\section{An Example Where the Measure Has a Negative Density}\label{section_11}

Let us present an example showing that the condition of non-negativity of the measure
density is essential, i.e., it does not follow from other stationarity
conditions. Consider the following problem:
\begin{equation}\label{sample_1}
	\begin{cases}
		\, z_1(T) + \left(z_1(0) - \widehat{z}_1\right)^2 + \left(z_2(0) - \widehat{z}_2 \right)^2 + \\[3pt]
		\phantom{rrrrrrrrrrrr} +\left(x(0) - \widehat{x}_0\right)^2 + \left(x(T) - \widehat{x}_T\right)^2 \to\min,\\[3pt]
		\begin{aligned}
			\dot z_1 &= (z_2 - a)(z_2 - b) x, \qquad \dot z_2 = 1,\\[3pt]
			\dot x &= u, \qq\q x \geq 0, \qquad\;\; |u| \leq 1.
		\end{aligned}
	\end{cases}
\end{equation}
Here $z = (z_1, z_2)\in \R^2,$ the parameters $0 < a < b < T$ are fixed, while the
parameters $\widehat{z}_1,\, \widehat{z}_2,\; \widehat{x}_{0},\, \widehat{x}_{T}$
are also fixed and will be defined below. The function $f =(f_1,f_2)= ((z_2- a)(z_2 -b)x, \,1)),$
thus $f'_x = ((z_2 - a)(z_2 - b),\,0)).$ The endpoints of the trajectory are free. \smallskip

Consider a trajectory with $u^0 = (-1, 0, 1)$ on the intervals $[0,a],$ $[a,b],$
$[b,T],$ respectively, $z_1^0(0) = 0,$ $z_2^0(t) \equiv t,$ and $x^0(t) = 0$
on $[a,b].$ Thus, $x^0(t) = a - t> 0$ for $t < a$ and $x^0(t) = t - b$ for $t > b.$
Check, whether stationarity conditions \eqref{adjoint_eq_5}--\eqref{stat_cond} hold.
\smallskip

The extended Pontryagin function is
$$
\overline{H} = \psi_{z_1} (z_2 - a) (z_2 - b) x + \psi_{z_2} +
\widetilde{\psi}_x u + \dot \mu x - h (u^2 - 1),
$$
where, in view of the complementary slackness conditions,
$\dot\mu =0$ outside of $[a,b],$ and $h=0$ on $[a,b].$

The adjoint equations and transversality conditions are as follows:
\begin{equation}\label{samp_1_adj_sys}
	\begin{cases}
		%\begin{aligned}
		\q \dot{\psi}_{z_1} = 0, \\[3pt]
		\q \dot{\psi}_{z_2} = - \psi_{z_1}\left(2 z^0_2 - a - b\right)x^0, \\[3pt]
		\q \dot{\widetilde{\psi}}_x = -\psi_{z_1} (z^0_2 - a) (z^0_2 - b) - \dot\mu,  \\[4pt]
		\;\psi_{z_1}(0) = 2 (z^0_1(0) - \widehat{z}_1), \qq
		\;\,\psi_{z_1}(T) = - 1,\\[3pt]
		\;\psi_{z_2}(0) = 2 \left(z^0_2(0) - \widehat{z}_2\right), \qq
		\psi_{z_2}(T) = 0,\\[3pt]
		\;\widetilde{\psi}_{x}(0) =  2 \left(x^0(0) - \widehat{x}_0\right), \qq
		\;\widetilde{\psi}_{x}(T) = - 2 \left(x^0(T) - \widehat{x}_T\right).
		%\end{aligned}
	\end{cases}
\end{equation}
By the first equation, $\psi_{z_1} \equiv - 1,$ hence, if we set
$\widehat z_1 = 1/2,$ the transversality condition for $\psi_{z_1}$
is satisfied. \smallskip

From \eqref{samp_1_adj_sys}, we get equations for $\psi_{z_2}$:
\begin{equation}
	\begin{cases}
		\begin{aligned}
			\dot{\psi}_{z_2} &= (2t - a - b)(a - t),\quad \mbox{ on } [0,a], &
			\quad \psi_{z_2}(0) &= - 2 \widehat{z}_2, \\[3pt]
			%&\dot{\widetilde{\psi}}_x & = (t - a)(t - b) \mbox{ Ð½Ð° } [0,a],\\
			\dot{\psi}_{z_2} &= 0 \qquad\qquad \qquad\qquad \quad \mbox{ on } [a, b],\\[3pt]
			\dot{\psi}_{z_2} &= (2t - a - b)(t - b), \quad \mbox{ on } [b,T], &
			\quad \psi_{z_2}(T) &= 0.\\
			%&\dot{\widetilde{\psi}}_x & = (t - a)(t - b)\mbox{ Ð½Ð° } [0,a]. \\
		\end{aligned}
	\end{cases}
\end{equation}\\[2pt]
Solving the initial value problems on $[0,a]$ and $[b,T],$ we get
\begin{equation}
	\begin{aligned}
		\psi_{z_2} &= - \frac{2}{3}t^3 + \frac{3 a + b}{2} t^2 - a (a + b)t
		- 2 \widehat{z}_2 \qquad \mbox{ on } [0,a],\\[3pt]
		\psi_{z_2} &= \frac{2}{3}\left(t^3 - T^3\right) -
		\frac{a + 3 b}{2} \left(t^2 - T^2\right) + b (a + b) \left(t - T\right)
		\; \mbox{ on } [b,T],\\
	\end{aligned}
\end{equation}
and, since $\psi_{z_2} $ is continuous everywhere and constant on $[a,b],$ it
should satisfy the equality $\psi_{z_2}(a - 0) = \psi_{z_2}(b + 0),$ i.e.,
\begin{multline}
	-\frac{2}{3}a^3 + \frac{3 a + b}{2} a^2 - a(a + b)a - 2 \widehat{z}_2\;
	= {} \\ {} = \frac{2}{3}\left(b^3 - T^3\right) -
	\frac{a + 3 b}{2} \left(b^2 - T^2\right) + b (a + b) \left(b - T\right).
\end{multline}
Obviously, there exists such $\widehat z_2$ that it holds.
Fix this $\widehat z_2$. \smallskip

Similarly, for $\widetilde{\psi}_x$ we get from \eqref{samp_1_adj_sys}:
\begin{equation}\label{psitil}
	\begin{cases}
		\dot{\widetilde{\psi}}_x  = (t - a)(t - b)- \dot\mu, \qquad \\[4pt]
		\widetilde{\psi}_x(0)= 2(a -\widehat{x}_0),\quad\;
		\widetilde{\psi}_x(T)= -2(T-b -\widehat{x}_T).\quad
	\end{cases}
\end{equation}
Solving the initial value problems on $[0,a]$ and $[b,T]$ with $\dot\mu=0,$
we get
%\begin{equation}
$$
\begin{cases}
\begin{aligned}
\widetilde{\psi}_x &= \frac{t^3}{3} -\frac{a + b}{2} t^2 +\; ab\,t + 2(a -\widehat{x}_0)
\quad \mbox{ on } [0,a],\\[4pt]
\widetilde{\psi}_x &= \frac{t^3 - T^3}{3} - \frac{a + b}{2} (t^2 - T^2)
+ ab(t -T) - 2(T - b -\widehat{x}_T) \quad \mbox{on } [b,T].\\
\end{aligned}
\end{cases}
$$
%\end{equation}
The condition $\overline{H}_u \equiv 0,$ i.e., $\widetilde{\psi}_x \equiv 2 h u^0,$
implies $\widetilde{\psi}_x \equiv 0$ on $[a, b].$ According to \eqref{jump_cond},
$\Delta\widetilde{\psi}_x(a) \leq 0,$ hence $\widetilde{\psi}_x(a-0) \ge 0.$
If $\widetilde{\psi}_x(a-0) > 0,$ then $h<0$ in a left neighborhood of $a,$
a contradiction with $h\ge 0.$ Therefore, $\widetilde{\psi}_x(a - 0) = 0.$
Similarly, we get $\widetilde{\psi}_x(b + 0) = 0,$ i.e., $\widetilde{\psi}_x$
has no jumps at $t= a$ and $t=b.$ \smallskip

The fulfillment of the obtained equalities is equivalent to the following
linear relations on the parameters $\widehat{x}_0,\, \widehat{x}_T$:
\begin{equation}
	\begin{cases}
		\begin{aligned}
			&\frac{a^3}{3} - \frac{a + b}{2} a^2 + a^2b + 2 (a - \widehat{x}_0)=0,\\[4pt]
			&\frac{b^3 - T^3}{3} - \frac{a + b}{2} (b^2 - T^2)  ab^2
			-2 ((T-b) - \widehat{x}_T)=0.
		\end{aligned}
	\end{cases}
\end{equation}
Obviously, such $\widehat{x}_0,\, \widehat{x}_T$ do exist. Fix these values.

Finally, from \eqref{psitil} it follows that $\dot{\widetilde{\psi}}_x > 0$
on $(0,a)$ and $(b,T),$ so $\widetilde{\psi}_x < 0$ on $[0,a)$ and
$\widetilde{\psi}_x > 0$ on $(b,T],$ and then the condition
$\widetilde{\psi}_x = 2 h u^0$ implies that $h(t) > 0$ on these intervals.
Thus, for the chosen parameters of problem and for the examined trajectory,
there exists a unique collection of multipliers satisfying all the conditions
of Theorem \ref{stat} except \eqref{nonnegativity_cond}. Here, condition \eqref{psitil}
implies that $\dot{\mu} = (t - a)(t - b) < 0$ on $(a,b),$ which contradicts
the condition \eqref{nonnegativity_cond}. Thus, the last condition
does not follow from the others, and the examined trajectory does not provide
the extended weak minimality.

%%------------------------------
\section{Generalization of the Obtained Result}\label{section_12}

An important feature of problem \eqref{problem_A} is that the state constraint
has the form $x \geq 0,$ i.e., it is imposed only on one state coordinate.
Let us show how it is possible to use the above result to formulate stationarity
conditions in a more general
\begin{equation}\label{problem_C}
	\mbox{{\textbf{Problem C:}}}\quad
	\begin{cases}
		\begin{aligned}
			&\dot y = f(y, u),\; && J_C:= J\left(y(0),y(T)\right)\to\min,\\[3pt]
			&\varphi\left(u(t)\right)\leq 0,\; && \varPhi\left(y(t)\right) \geq 0.
		\end{aligned}
	\end{cases}
\end{equation}
Here  $y\in \R^{n + 1},$ $u\in \R^{m},$ the state variable $y(\cdot)$ is
absolutely continuous, and the control $u(\cdot)$ is measurable bounded
functions. We assume that the data functions $f,\,\varphi,$ and $\Phi$
are defined and twice continuously differentiable on an open subset
$\mathcal{Q}\subset\R^{n + 1 +m}.$ \ssk

As before, we suppose that the reference process $w^0= (y^0,u^0)$ is such
that the trajectory $y^0(t)$ touches the state boundary only on a segment
$[t^0_1, t^0_2],$ where $0 < t^0_1 < t^0_2 < T.$  In other words, the interval
$\Delta:= [0,T]$ is divided into three parts $\Delta_1:= [0,t^0_1],$
$\Delta_2:= [t^0_1, t^0_2],$ and $\Delta_3:= [t^0_2, T],$ such that
$\Phi(y^0(t)) > 0$ on $[0, t^0_1),$ $\Phi(y^0(t)) = 0$ on $\Delta_2,$
and $\Phi(y^0(t)) > 0$ on $(t^0_2,T].$ The control $u^0(t)$ is continuous
on $\Delta_1\,, \Delta_3\,,$ Lipschitz continuous on $\Delta_2,$
and, moreover, $\varphi_s(u^0(t))<0$ on $\Delta_2$ for all $s,$
and the landing to the state boundary and the leaving it occurs with
nonzero time derivatives:
\begin{equation}\label{y12}
	\begin{aligned}
		\dot \Phi\left(y^0(t^0_1 - 0)\right) &= \;
		\Phi'\left(y^0(t^0_1)\right)f\left(y^0(t^0_1),\, u^0(t^0_1 - 0) \right) < 0,
		\\[2pt]
		\dot \Phi\left(y^0(t^0_2 + 0)\right) &=\;
		\Phi'\left(y^0(t^0_2)\right)f\left(y^0(t^0_2),\, u^0(t^0_2 + 0) \right) > 0.
	\end{aligned}
\end{equation}
As before, we assume that the gradients $\varphi'_i(u^0(t)),\;\, i \in I(u^0(t))$
are positive independent for all $t\in \Delta_1 \cup\Delta_3\,,$
and $ \Phi'(y^0(t))f_u(y^0(t),\, u^0(t))\ne 0$ on $\Delta_2$.
%\smallskip

%%----------------------------------
\section{Reduction of Problem C to Problem A}\label{section_13}

We accept the following technical \ms

{\bf Assumption C.}\q There exist an open subset $\,\Omega \subset \R^{n+1}\,$
containing\, the curve $y^0(t),\;\, t\in [0,T],$ and twice continuously differentiable
functions $P_i: \Omega \to\R,\;\, i=1,\ldots, n,$ such that the gradients
$P_1'(y),\ldots,P_n'(y),\,\Phi'(y)$ are linearly independent at any point
$y\in \Omega,$ and, moreover, the mapping $F: \Omega \to \R^{n+1}$
defined by  \vadjust{\kern-8pt}
\begin{equation}\label{F}
	F(y) :=
	\left(
	\begin{aligned}
		P(y)\\
		\Phi(y)
	\end{aligned}
	\right) =
	\left(
	\begin{aligned}
		P_1&(y)\\
		&\vdots\\
		P_n&(y)\\
		\Phi&(y)
	\end{aligned}
	\right)
\end{equation}
is an injection. In other words, $F$ realizes a nondegenerate change
of variables in $\Omega$:
\begin{equation}\label{yTozx}
	y \mapsto (z,x),\qquad z = P(y) \in \R^n,\qquad x = \Phi(y) \in \R^1.
\end{equation}
Herewith, $det\,F'(y^0(t)) \ne 0,$ the set $Q = F(\Omega)$ is also open,
and there exists a inverse mapping $G: Q \to \Omega,$ $(z, x)\mapsto y,$
so that
\begin{equation}\label{pphi}
	G(P(y), \Phi(y)) \;=\; y \qquad \forall y \in \Omega.
\end{equation}
In what follows, we will always assume that $y,\,z,\,x$ satisfy the
following relations
$$
y =G(z,x),\qquad z=P(y),\qquad x=\Phi(y).
$$
Note that differentiation of \eqref{pphi} yields the equality
\begin{equation}\label{gzpp}
	G'_z(z,x)\, P'(y)\; +\; G'_x(z,x)\, \Phi' (y)\; =\; E_{n+1}\,,
\end{equation}
where the right hand part is the identity matrix of dimension $n+1$.
\smallskip

\begin{remark} It is sufficient to assume that $\Omega$ contains not the
	entire curve $y^0(t),$ $t\in [0,T],$ but only part of it for $t \in\Delta_2.$
	Then, by extending the definition of the function $P$ out of $\Omega,$
	one can reduce the situation to the case of $\Omega$ containing the entire
	curve $y^0(t).$ Here we do not dwell on the corresponding technical details.
	Note only that Assumption C is really satisfied in all reasonable, especially
	applied, problems with state constraints.
\end{remark}

Obviously, the dynamics of state variables $z,\,x$ obeys the system
\begin{equation}
	\dot z  = P'(y) f(y, u),\qquad  \dot x = \Phi'(y) f(y, u),
\end{equation}
therefore, problem \eqref{problem_C} in these new variables transforms
to the following problem of type \eqref{problem_A} on the same time interval
$[0,T]$:
\begin{equation}\label{pro3}
	\mbox{{\textbf{Problem D:}}}\quad
	\begin{cases}
		\begin{aligned}
			&\dot z = P'\left(G(z,x)\right)f\left(G(z,x),u\right), \\[2pt]
			&\dot x = \Phi'\left(G(z,x)\right)f\left(G(z,x),u\right),\\[2pt]
			&J_D:=J\left(G\left(z(0), x(0)\right), G\left(z(T),x(T)\right)\right) \to \min,\\[2pt]
			&\varphi_i\left(u(t)\right) \leq 0, \quad i = 1,\dots, d(\varphi),\\[2pt]
			&x(t) \geq 0.\\[2pt]
		\end{aligned}
	\end{cases}
\end{equation}

To each process $w = (y(t), u(t))$ of problem C\, one can associate a
process $\gamma = (z(t), x(t), u(t))$ of problem D, and vice versa.
Obviously, the process $w^0$ provides the extended weak minimality in
problem C if and only if the corresponding process $\gamma^0$ provides
the extended weak minimality in problem~D.

Therefore, we can use the fact that the process $\gamma^0$ satisfies
the stationarity conditions given in Theorem \ref{stat}.

%%-------------------------------
\section{Stationarity Conditions for Problem C}\label{section_14}

In further transformations, we have to differentiate vector-valued and
matrix-valued functions w.r.t a vector argument. To avoid cumbersome formulas
in the coordinate form, let us accept the following notation. If $T(z)$
is any tensor of a given rank (in particular, a vector or a matrix), every
element $\theta(z)$ of which is a smooth function of $z \in\R^n,$ then its
directional derivative along a vector $\bar z \in\R^n$ will be denoted as
$T'(z)\,\bar z.$ The last one is still a tensor of the same rank and dimension,
whose elements $\theta'(z)\,\bar z = \sum_{i=1}^n \theta'_{z_i}(z)\,\bar z_i$
are the scalar directional derivatives of the corresponding elements
$\theta(z)$ along the vector $\bar z$.
\smallskip

According to Theorem \ref{stat}, if the process $\gamma^0 = (z^0(t), x^0(t), u^0(t))$
provides the extended weak minimality in problem D, then there exist a Lipschitz
continuous adjoint variable $\psi_z(t)$ $(n-$dimensional row vector) on $[0,T],$
a constant $c,$ scalar functions $\mu(t)$ and $\psi_x(t),$ Lipschitz continuous on
each interval $\Delta_i,$ $i=1,2,3,$ such that $d\mu(t) \ge 0,$ and a measurable
bounded function $h(t) \ge 0,$ which generate the Pontryagin function %\vad
$$
{\cal H} = \Bigl(\psi_z P'\left(G(z,x)\right) + \psi_x \Phi'\left(G(z,x)\right)\Bigr) f\left(G(z,x),u\right)
$$
and the extended Pontryagin function
$$
\overline{\cal H} = \Bigl(\psi_z P'(G(z,x)) +
\psi_x \Phi'(G(z,x)\Bigr) f\left(G(z,x),u\right)\, +\, \dot \mu x - h \varphi,
$$
such that the following conditions hold:\ssk

\noindent
complementary slackness
\begin{equation}\label{comp_slack_2}
	\dot \mu(t)\, x^0(t) = 0,\; \quad h(t)\, \varphi(u^0(t)) = 0
	\quad \mbox{ a.e.\, on }\; [0,T],
\end{equation}
adjoint equations
\begin{equation}\label{eq63a}
	\begin{cases}
		\begin{aligned}
			- &\dot {\psi}_z\, \overline{z}   =  \Bigl(\psi_z P''\left(G(z^0, x^0)\right)\, +
			\psi_x \Phi''\left(G(z^0, x^0)\right)\Bigr) \\ &
			\phantom{rrrrrrrrrrrr} \times\left( G'_z(z^0, x^0) \overline{z}\right) f(G(z^0, x^0),u^0)\; +\\[2pt]
			&+ \Bigl(\psi_z P'\left(G(z^0, x^0))\right)\, + \psi_x \Phi'\left(G(z^0, x^0)\right)\Bigr) \\ &
			\phantom{rrrrrrrrrrrr} \times f'_y(G(z^0, x^0),u^0) \left( G'_z(z^0, x^0) \overline{z}\right),\\[6pt]
		\end{aligned}
	\end{cases}
\end{equation}

\begin{equation}\label{eq63b}
	\begin{cases}
		\begin{aligned}
			- &\dot {\psi}_x\, \overline{x}  =  \Bigl(\psi_z P''\left(G(z^0, x^0)\right)\, +
			\psi_x \Phi''\left(G(z^0, x^0)\right)\Bigr) \\ &
			\phantom{rrrrrrrrrrrr} \times\left( G'_x(z^0, x^0) \overline{z}\right) f(G(z^0, x^0),u^0)\; +\\[2pt]
			&+ \Bigl(\psi_z P'\left(G(z^0, x^0))\right)\, + \psi_x \Phi'\left(G(z^0, x^0)\right)\Bigr) \\ &
			\phantom{rrrrrrrrrrrr} \times f'_y(G(z^0, x^0),u^0) \left( G'_x(z^0, x^0) \overline{z}\right),\\[6pt]
			%       - &\dot {\psi}_x\, \overline{x}   =  \Bigl(\psi_z P''\left(G(z^0, x^0)\right) + \\ &+
			%       \psi_x \Phi''\left(G(z^0, x^0)\right)\Bigr)\left( G'_x(z^0, x^0) \overline{x}\right) f(G(z^0, x^0),u^0)\\
			%       &+\Bigl(\psi_z P'\left(G(z^0, x^0))\right) + \\ &+ \psi_x \Phi'\left(G(z^0, x^0)\right)\Bigr) f'_y(G(z^0, x^0),u^0) \left( G'_x(z^0, x^0) \overline{x}\right) + \dot\mu \overline{x} ,\\
		\end{aligned}
	\end{cases}
\end{equation}
(these equalities hold for any ``test'' constant vectors
$\overline{z}\in \R^n$ and $\overline{x} \in \R^1),$ \\[4pt]
transversality conditions
\begin{equation}\label{trans_2}
	%\begin{cases}
	\begin{aligned}
		\psi_z(0) &= J'_{y(0)}\,G'_z(z^0(0),x^0(0)),& \;\;
		\psi_z(T) &= - J'_{y(T)}\,G'_z(z^0(T),x^0(T)),\\[2pt]
		\psi_x(0) &= J'_{y(0)}\,G'_x(z^0(0),x^0(0)),& \;\;
		\psi_x(T) &= - J'_{y(T)}\,G'_x(z^0(T),x^0(T)),
	\end{aligned}
	%\end{cases}
\end{equation}
jump conditions for the adjoint variable $\psi_x$
\begin{equation}\label{jump_2}
	\Delta \psi_x(t_1^0) = - \Delta \mu(t_1^0) \leq 0, \qquad
	\Delta \psi_x(t_2^0) = - \Delta \mu(t_2^0) \leq 0,
\end{equation}
the energy conservation law
\begin{equation}\label{energy_conservation_2}
	{\cal H}(z^0(t),x^0(t),u^0(t))\, =\, c,
\end{equation}
and stationatity condition w.r.t. control
\begin{equation}\label{staunew}
	\Bigl(\psi_z P'\left(G(z^0, x^0)\right) +
	\psi_x \Phi'\left(G(z^0, x^0)\right)\Bigr) f'_u\left(G(z^0, x^0)),u^0\right) -
	h \varphi_u'(u^0) = 0. \\[4pt]
\end{equation}

Now, rewrite the obtained conditions in terms of problem C. First, denote
\begin{equation}\label{psiy}
	\psi_y =\; \psi_z P'(G(z^0, x^0)) + \psi_x \Phi'(G(z^0, x^0)).
\end{equation}
This is a row vector of dimension $n+1.$ Then, since $G(z^0, x^0)= y^0,$
condition \eqref{staunew} takes the form
\begin{equation}\label{Hyt}
	\psi_y\, f'_u (y^0, u^0) - h\, \varphi'_u (u^0) = 0.
\end{equation}

Further, multiplying $\psi_y$ by a test (constant) vector $\overline{y} \in \R^{n+1},$
we get a scalar function
$$
\psi_y\,\overline{y}\;  =\; \psi_z\, P'(G)\overline{y}\,  + \psi_x\,
\Phi'(G)\overline{y}
$$
(for short, we drop the arguments of $G$ and $f$),
which time derivative is
\begin{equation}\label{eq70}
	- \dot{\psi}_y\, \overline y\; = -\dot \psi_z (P'(G)\bar y) - \dot\psi_x (\Phi'(G)\bar y)\;
	- \psi_z (P'(G))^\bullet\overline{y}\,  - \psi_x (\Phi'(G))^\bullet \overline{y},
\end{equation}
{where $(...)^\bullet$ denotes the time derivarive of the
	function in brackets.}
\smallskip

Let us write the first two terms of this expression in view of equations
\eqref{eq63a} and \eqref{eq63b} for $\bar z = P'(G)\bar y,$
$\; \bar x = \Phi'(G)\bar y$:
\begin{equation}\label{eq71}
	\begin{array}{c}
		-\dot \psi_z (P'(G)\bar y) =\;
		\Bigl(\psi_z P''(G) + \psi_x \Phi''(G)\Bigr) \Bigl(G'_z\cdot (P'(G)\bar y)\Bigr) f\;+ \\[4pt]
		\; +\Bigl(\psi_z\, P'(G) + \psi_x\, \Phi'(G)\Bigr) f'_u  \Bigl(G'_z\cdot (P'(G)\bar y)\Bigr),
	\end{array}
\end{equation}

\begin{equation}\label{eq72}
	\begin{array}{c}
		- \dot\psi_x (\Phi'(G)\bar y) =\;
		\Bigl(\psi_z P''(G) + \psi_x \Phi''(G)\Bigr) \Bigl(G'_x\cdot (\Phi'(G)\bar y)\Bigr) f\;+ \\[4pt]
		\; +\Bigl(\psi_z\, P'(G) + \psi_x\, \Phi'(G)\Bigr) f'_u  \Bigl(G'_x\cdot (\Phi'(G)\bar y)\Bigr)\;
		+ \;\dot\mu\, \bar x.
	\end{array}
\end{equation}
The other two terms of \eqref{eq70} in view of identities $\dot G = \dot y = f$
are equal to
\begin{equation}\label{eq73}
	%\begin{array}{c}
	- \psi_z (P'(G))^\bullet\overline{y}\,  - \psi_x (\Phi'(G))^\bullet \overline{y}\; =
	\; - \Bigl(\psi_z P''(G) + \psi_x \Phi''(G)\Bigr) f\, \bar y.
	%\end{array}
\end{equation}

Summing up the right parts of equalities \eqref{eq71}--\eqref{eq73},
we get
\begin{equation}\label{eq75}
	\begin{array}{c}
		- \dot{\psi}_y\, \overline{y} =\; \Bigl(\psi_z P''(G) +
		\psi_x \Phi''(G)\Bigr)\,\bar y\, f\; + \\[4pt]
		\;+\, \psi_y\,f'_y\,\bar y\; + \;\dot\mu\, \bar x\; -
		\Bigl(\psi_z P''(G) + \psi_x \Phi''(G)\Bigr) f\, \bar y.
	\end{array}
\end{equation}
Note that the matrix $\psi_z P''(G) + \psi_x \Phi''(G)$ is the second
derivative of the scalar function $\psi_z P(G) + \psi_x \Phi(G),$ hence it
is symmetric. Therefore, the first and the last terms in the right hand
part of obtained expression (which differ only in the positions of multipliers
$\overline{y}$ and $f$) cancel each other, and in view of relation
$\bar x = \Phi'(G)\bar y,$ equation \eqref{eq75} takes the form
$$
- \dot{\psi}_y\, \overline{y}\; =\; \psi_y\,f'_y\,\bar y\; +
\;\dot\mu\,\Phi'(G)\,\bar y,
$$
whence, since the test vector $\bar y \in \R^{n+1}$ is arbitrary, we get
\begin{equation}\label{cost-y}
	- \dot{\psi}_y \; =\; \psi_y\,f'_y(y,u)\; + \;\dot\mu\,\Phi'(y).
\end{equation}
If we introduce the Pontryagin function $H = \psi_y f(y, u)$ and the
extended Pontryagin function $\overline{H} =
\psi_y f(y, u) + \dot\mu \Phi(y) - h \varphi(u)$ for problem C,
then equalities \eqref{Hyt} and \eqref{cost-y} transform to
$ \overline H_u =0$ and $-\dot{\psi}_y = \overline H_y\,$ respectively.
\smallskip

According to \eqref{jump_2}, the function $\psi_y$ has jumps
at the points $t_1^0,\, t_2^0$:
\begin{equation}\label{jumps_psi_y}
	\begin{aligned}
		\Delta \psi_y(t_1^0) &=\; \Delta \psi_x(t_1^0)\, \Phi'\left(y(t_1^0)\right) =\;
		- \Delta \mu(t_1^0)\, \Phi'\left(y(t_1^0)\right), \\[4pt]
		\Delta \psi_y(t_2^0) &=\; \Delta \psi_x(t_2^0)\, \Phi'\left(y(t_2^0)\right) =\;
		- \Delta \mu(t_2^0)\, \Phi'\left(y(t_2^0)\right).
	\end{aligned}
\end{equation}
The transversality conditions for $\psi_y$ take the form
\begin{equation}\label{trans_psi_y}
	\begin{aligned}
		\psi_y(0) =\; J'_{y(0)} G'_z(0) P'\left(y(0)\right) + J'_{y(0)} G'_x(0) \Phi'\left(y(0)\right)\;
		= J'_{y(0)},\\[4pt]
		\psi_y(T) =\; - J'_{y(T)} G'_z(0) P'\left(y(0)\right) - J'_{y(T)} G'_x(T) \Phi'\left(y(T)\right)\;
		= - J'_{y(T)}.
	\end{aligned}
\end{equation}
Finally, the complementary slackness conditions and the energy conservation
law are rewritten automatically in terms of problem C.
\smallskip

%%-----------------------------------
Summarizing our findings, we come to the following

\begin{theorem}\label{stat_2} \quad
	Let $w^0 = (y^0(t), u^0(t)$ be an admissible process such that $\Phi(y^0(t)) = 0$
	on $\Delta^0_2:=\left[t^0_1,t^0_2\right],$ $\Phi(y^0(t)) > 0$ on
	$[0,T]\setminus \Delta^0_2,$ ${\;\varphi_i(u^0(t)) < 0}$ on $\Delta_2\,,$
	assumption \eqref{y12} holds, and let this process provide the extended
	weak minimality in problem C.\, Then there exist a constant $c,$ functions
	$\psi_y(t),$ $\mu(t)$ Lipschitz continuous on every interval
	$\Delta_i,\; i=1,2,3,$ and a measurable bounded function $h(t),$
	which generate the Pontryagin function  %\vad
	$$
	H(\psi_y,\,y, u) =\; \psi_y f(y, u),
	$$
	and the extended Pontryagin function  %\vad
	$$
	\overline{H} =\; \psi_y f(y, u) + \dot\mu\, \Phi(y) - h \varphi(u),
	$$
	such that the following conditions hold:
	\begin{enumerate}[(a)]
		\item  non-negativity conditions
		\begin{equation}\label{nonnegativity_cond_C}
			\begin{array}{c}
				\dot\mu(t)\ge 0\; \mbox{ a.e. on } \Delta^0_2, \quad     \Delta\mu(t_1^0) \ge 0,
				\q \Delta\mu(t_2^0) \ge 0, \\[6pt]
				h(t) \ge 0\q \mbox{ a.e.\, on }\, [0,T],
			\end{array}
		\end{equation}
		\item complementary slackness
		\begin{equation}\label{comp_slack_cond_C}
			\dot\mu(t)\Phi\left(y^0(t)\right) = 0, \quad\; h(t) \varphi(u^0(t)) \;
			\mbox{ a.e. on }\, [0,T]
		\end{equation}
		\item adjoint equation
		\begin{equation}\label{adjoint_eq_C}
			-\dot \psi_{y} = \; \overline H_y =\;  \psi_y f'_y(y^0, u^0) + \dot\mu\, \Phi'(y^0),
		\end{equation}
		\item transversality conditions
		\begin{equation}\label{tr-y}
			\psi_y(t_0) =\; J'_{y(0)}, \qquad \psi_y(T) =\; - J'_{y(T)},
		\end{equation}
		\item jumps conditions for the adjoint variable
		\begin{equation}\label{jump_cond_C}
			\Delta\psi_y(t_1^0) = - \Delta\mu(t_1^0)\, \Phi'\left(y^0(t_1^0)\right), \quad
			\Delta\psi_y(t_2) = - \Delta\mu(t_2)\, \Phi'\left(y^0(t_2)\right),
		\end{equation}
		\item energy conservation law
		\begin{equation}\label{energy_conservation_C}
			H(\psi_y(t),y^0(t),u^0(t)) = c,
		\end{equation}
		\item and stationarity condition w.r.t. control
		\begin{equation}\label{stat_cond_C}
			\overline{H}'_u(\psi_y(t),y^0(t),u^0(t)) = 0\quad \mbox{ a.e. on }\; [0,T].
		\end{equation}
	\end{enumerate}
\end{theorem}

%%--------------------------------------
\begin{remark}\,
	The performed transformation $y \mapsto (z,x)$ is a particular case of the
	general one-to-one change of variables $w =F(y),\; y= G(w),$ under which
	problem C transforms to the following
	\begin{equation}\label{probl_3}
		\mbox{{\textbf{Problem E:}}}\quad
		\begin{cases}
			\begin{aligned}
				&\dot w = F'(G(w))\,f(G(w),u), \\[2pt]
				&J_E:= J\left(G(w(0)), G(w(T))\right) \to \min,\\[2pt]
				&\varphi (u(t)) \leq 0, %\quad i = 1,\dots,dim(\varphi),
				\\[2pt]
				& \Phi(G(w(t)) \geq 0. %\\[6pt]
			\end{aligned} \\[6pt]
		\end{cases}
	\end{equation}
	Clearly, the extended weak minimality at a process  $(y^0,u^0)$ in
	problem C corresponds to that at the process $(w^0 =F(y^0),\,u^0)$ in
	problem E.  The multipliers $\alpha_0,\, h,\, \mu$ in both problems are
	the same, the extended Pontryagin functions for problems C and E are,
	respectively,
	$$ \begin{array}{l}
	\overline H^C =\; \psi^C f(y,u) - h \varphi(u) + \dot\mu^C\,\Phi(y), \\[4pt]
	\overline H^E =\; \psi^E F'(G(w))\, f(G(w),u) - h \varphi(u) + \dot\mu^E\,\Phi(G(w)),
	\end{array}
	$$
	while the adjoint variables are connected by the following equality:
	$$ \psi^C(t) =\; \psi^E(t)\, F'(y^0(t)).
	$$
	The proof of this assertion is left to the reader as an excercise.
	\ssk
	
	In the case of problem D, we have $w=(z,x)$ and $F = (P,\Phi),$ hence
	$$ \psi^C(t) =\; \psi^E_z(t)\, P'(y^0(t)) + \psi^E_x(t)\, \Phi'(y^0(t)),$$
	i.e., we get exactly formula \eqref{psiy}.
\end{remark} %\smallskip

\begin{remark}\,
	For simplicity, we considered problem A\, with free endpoints of the
	trajectory. If they are restricted by terminal constrains
	$$\xi_k(z(0),x(0),\,x(T), z(T)) \leq 0,\quad\; \eta_j(z(0),x(0),\,x(T), z(T)) = 0,
	$$
	then, to obtain stationarity conditions, one should replace the cost $J$
	by the endpoint Lagrange function
	$l = \alpha_0 J + \sum_k \alpha_k \xi_k + \sum_j \beta_j \eta_j$
	(with corresponding multipliers)  and then apply Theorem \ref{stat}.\,
	%or Theorem \ref{stat_2} in case of problems A or C, respectively.
	The same concerns problem C.
\end{remark} %\smallskip

\begin{remark}
	We suppose that the state constraint in problem \eqref{problem_A} or
	\eqref{problem_C} is of first order, and the reference trajectory lands
	on the state boundary with a nonzero first time derivative, i.e., satisfies
	conditions  \eqref{x12} or \eqref{y12}, respectively.
	Perhaps, the same approach would also work in the case of higher order
	state constraints, if the reference trajectory lands on the state boundary
	with a nonzero time derivative of the corresponding order.
	Obviously, the technique would be then more complicated.
\end{remark}

\section{Conclusions}
We consider a specific class of optimal control problems with a single
state constraint of order 1\, and a specific trajectory in it. Basing on
the approach by R.V.~Gamkrelidze, consisting in differentiating the state
constraint along the boundary subarc and reducing the original problem to
a problem with mixed control-state constraints, we obtain the full system
of stationarity conditions in the form of A.Ya.~Dubovitskii and A.A.~Milyutin,
including the sign definiteness of the measure, a multiplier at the state
constraint. To obtain these conditions, we propose an approach of two-stage
varying. At the first stage, we consider only those variations, which preserve
a constant value of the state constraint along the boundary interval, and
obtain preliminary, incomplete optimality conditions. At the second stage,
we take into account the remaining variations, concentrated on the boundary
interval, and obtain the sign definiteness of the measure, thus specifying
the stationarity conditions. Two illustrative examples are given, one showing
that the condition of non-negativity of the measure density is essential
and another with nonzero atoms of the measure at the junction points.

\begin{acknowledgements}
	This research was partially supported by the Russian Foundation for Basic
	Research under grant No. 16-01-00585.\, The authors thank Nikolai Osmolovskii
	for useful discussions and the anonymous referees for valuable remarks.
\end{acknowledgements} 
\bigskip

%%-------------------------------------------------------


\begin{thebibliography}{99}
	
	\bibitem{DM65}
	{Dubovitskii, A.\,Ya., Milyutin, A.\,A.:}
	\href{https://doi.org/10.1016/0041-5553(65)90148-5}{Extremum problems in the presence of restrictions.}
	{USSR Comput. Math. and Math. Phys.}
	{5(3),} {1--80} {(1965)}
	
	\bibitem{Gir}
	{Girsanov I.V.:}
	\href{http://www.springer.com/us/book/9783540058571}{Lectures on Mathematical Theory of Extremum Problems.}\,
	{Springer-Verlag Berlin, Heidelberg}
	{(1972)}
	
	
	\bibitem{IT}
	{Ioffe, A.D., Tikhomirov, V.M.:}
	{Theory of extremal problems.}
	{North-Holland Publishing Company,} {Amsterdam, New Yourk, Oxford}
	{(1974)}
	
	
	\bibitem{Pont}
	{Pontryagin, L. S., Boltyanskii, V. G., Gamkrelidze, R. V., Mishechenko, E. F.:}
	{The Mathematical Theory of Optimal Processes.}
	{John Wiley \& Sons,}
	{New York/London}
	{(1962)}
	
	\bibitem{HSV}
	{Hartl, F.\,H, Sethi, S.\,P., Vickson, R.\,G.:}
	\href{http://dx.doi.org/10.1137/1037043}{A survey of the maximum principles for optimal control problems
		with state constraints}.
	{SIAM Review}
	{37(2),}
	{181--218}
	{(1995)}
	
	
	\bibitem{AKP}
	{Arutyunov, A.\,V., Karamzin, D.\,Y., Pereira, F.\,L.:}
		{\href{http://dx.doi.org/10.1007/s10957-011-9807-5}{The Maximum Principle  for Optimal Control Problems with State Constraints
		by R.V. Gamkrelidze: Revisited}}.
	{J. Optim. Theory Appl.}
	{149(3),} {474--493} {(2011)}
	
	
	
	\bibitem{DO}
	{Dmitruk, A.\,V., Osmolovskii, N.\,P.:}
	{\href{http://dx.doi.org/10.3934/dcds.2015.35.4323}{Necessary conditions for a weak minimum in optimal control problems
		with integral equations on a variable time interval}}.
	{Discrete and Continuous Dynamical Systems}
	{35(9),} {4323--4343} {(2015)}
	
	\bibitem{DK}
	{Dmitruk, A.\,V., Kaganovich A.\,M.:}
	{\href{http://dx.doi.org/10.1016/j.sysconle.2008.05.006}{The Hybrid Maximum Principle is a consequence of Pontryagin Maximum Principle}}.
	{Systems \& Control Letters}
	{57(11),} {964--970} {(2008)}
	
	\bibitem{Den}
	{Denbow, C.\,H.:}
	A generalized form of the problem of Bolza.
	Contributions to the Calculus of Variations, 1933-1937,
	{The University of Chicago Press}, p. 449--484 {(1937)}
	
	
	\bibitem{Vol-Ost}
	{Volin, Yu.\,M., Ostrovskii, G.\,M.:}
	{Maximum principle for discontinuous systems and its
		application to problems with state constraints (In Russian).}
	{Izvestia Vuzov. Radiofzika}
	{12,} {1609--1621} {(1969)}
	
	\bibitem{AgMa}
	Augustin, D., Maurer, H.:
		{Second order sufficient conditions and sensitivity analysis for optimal
	multiprocess control problems}. Control and Cyb.,
	29, No. 1, pp. 11-31 (2000)
	
	
	\bibitem{MBKK}
	{Maurer, H., Buskens, C., Kim, J.-H.\,R.,
		Kaya, C.\,Y.:}
		{\href{http://dx.doi.org/10.1002/oca.756}{Optimization methods for the verification of second
		order sufficient conditions for bang-bang controls}}.
	{Optimal Control Applications and Methods}
	{26,} {129--156} {(2005)}
	
	
	\bibitem{ObRo}
	Oberle, H.J., Rosendahl, R.: 
	{On singular arcs in nonsmooth optimal
	control}. Control and Cybernetics, 37, no. 2, p. 429--450 (2008)
	
	
	\bibitem{DK2}
	{Dmitruk, A.\,V., Kaganovich A.\,M.:}
	{\href{http://dx.doi.org/10.1007/s10598-011-9096-8}{Maximum principle for optimal control problems with intermediate constraints}}.
	{Comput. Math. and Modeling}
	{22(2),} {180--215} {(2011)}
	
	
	\bibitem{LTJW}
	Liu, Y.,  Teo, K.L., Jennings, L.S., Wang, S.:
	{\href{http://dx.doi.org/10.1023/A:1022684730236}{On a class of optimal control problems with state jumps}}. {J. Optim. Theory Appl.},
	98, no. 1, p. 65--82 (1998)
	
	\bibitem{MO}
	{Milyutin, A.\,A, Osmolovskii, N.\,P.:}
	{Calculus of variations and optimal control.}
	{American Mathematical Society,} {Providence}  {(1998)}.
	
	
	\bibitem{MDO}
	{Milyutin, A.\,A., Dmitruk, A.\,V, Osmolovskii N.\,P.:}
	{Maximum principle in optimal control (Princip maksimuma v optimal'nom
		upravlenii, in Russian).}
	{Lomonosov Moscow State University, Faculty of Mathematics and Mechanics,}
	{Moscow}
	{(2004)}, available at
	\href{http://www.milyutin.ru/book eng.html}{www.milyutin.ru}
	
	\bibitem{SIC}
	{Dmitruk, A.\,V., Osmolovskii, N.\,P.:}
	{\href{http://dx.doi.org/10.1137/130921465}{Necessary conditions for a weak minimum in optimal control problems
		with integral equations subject to state and mixed constraints}.}
	{SIAM J. on Control and Optimization}
	{52(6),}
	{3437--3462}
	{(2014)}	
	
	
	\bibitem{ADMC}
	{Afanasyev, A.P., Dikusar, V.V., Milyutin, A.A., Chukanov, S.V.:}
	{Necessary condition in optimal control (Neobchodimoye uslociye
		v optimal'nom upravlenii, in Russian).}
	{Nauka,} {Moscow} {(1990)}, available at
	\href{http://www.milyutin.ru/book eng.html}{www.milyutin.ru}
	
	
	\bibitem{M77}
	{Maurer, H.:}{\href{http://dx.doi.org/10.1137/0315023}{On optimal control problems with bounded state variables
	and control appearing linearly}}.
	SIAM J. Control Optim., 15, no. 3, p. 345-362 (1977)
	
	
	
	\bibitem{BdV}
	{Bonnans, J.F., de la Vega, C.:} 
	{\href{http://dx.doi.org/10.1007/s11228-010-0154-8}{Optimal Control of State Constrained
	Integral Equations}}. Set-Valued Analysis, 18, p. 307-326 (2010)
	
	
	\bibitem{Pin-Sch}
	{de Pinho, M.R., Shvartsman, I.:} 
	{\href{http://dx.doi.org/10.3934/dcds.2011.29.505}{Lipschitz continuity of optimal control
	and Lagrange multipliers in a problem with mixed and pure state constraints}}.
	Discrete Continuous Dynamical Systems, Ser. A, 29, no. 2, 505--522 (2011)
	
	
	\bibitem{A-2016}
	{Arutyunov, A.\,V., Karamzin, D.\,Yu., Pereira, F.:}
	{\href{http://dx.doi.org/10.1134/S0081543816020036}{Conditions for the Absence of Jumps of the Solution to the Adjoint System
		of the Maximum Principle for Optimal Control Problems with State Constraints.}}
	{Proc. of the Steklov Institute of Mathematics}
	{292(1),} {27--35} {(2016)}
	
	
	
	%\selectlanguage{english}
	
\end{thebibliography}
\end{document}